\documentclass[11pt,a4paper]{amsart}

\usepackage{todonotes}
\usepackage{amsmath}
\usepackage{eufrak}
\usepackage{amsfonts}
\usepackage{amssymb}
\usepackage{amsthm}
\usepackage{color}
\usepackage[cp1250]{inputenc}
\usepackage{setspace}
\usepackage[T1]{fontenc}
\usepackage{geometry}
\usepackage{fancyhdr}
\usepackage{graphicx}
\usepackage{color}
\usepackage{enumerate}
\usepackage{cancel}
\usepackage[normalem]{ulem}
\newcommand{\sep}{\: \mid \: }

\newcommand{\N}{{\ensuremath{\mathbb N}}}

\newcommand{\E}{{\ensuremath{\mathbb E}}}
\newcommand{\Pro}{{\ensuremath{\mathbb P}}}

\newcommand{\R}{{\ensuremath{\mathbb R}}}

\newcommand{\D}{\nabla}
\newcommand{\ub}[1]{^{(#1)}}
\def\rd{{\bf d}}
\newcommand{\uu}[1]{^{\underline{#1}}}

\newcommand{\lv}{\left\lVert}
\newcommand{\rv}{\right\rVert}
 
\newcommand{\1}{\mathbf{1}}

\newcommand{\bg}{\bigotimes}
\newcommand{\eps}{\varepsilon}
\newcommand{\ii}{\mathbf{i}}
\newcommand{\ai}{a_\ii}
\newcommand{\pp}{\mathcal{P}}
\newcommand{\nor}[1]{\lv #1 \rv_{\mathcal{P}'\sep \mathcal{P} }}
\newcommand{\ind}[1]{{i_{I_{#1}}}}
\newcommand{\x}[1]{x^1_\ind{1} \cdots x^{#1}_\ind{#1}}
\newcommand{\al}[1]{\alpha_A \left( #1 \right)}
\newcommand{\bb}[1]{\beta_A \left(#1 \right)}
\newcommand{\bt}{\bigotimes}
\newcommand{\ten}{\otimes}

\newcommand{\aaa}{\alpha}
\newcommand{\be}{\beta}
\newcommand{\los}[2]{\bt_{k=#1}^{#2} \left(x^k (1-\1_I(k))+G^k\1_I(k) \right)}
\newcommand{\rh}{\rho^{\textbf{y},I}_A}
\newcommand{\sr}{\Delta^{\textbf{y},I}_A}
\newcommand{\xx}{\textup{\bf x}}
\newcommand{\xxx}{\tilde{\textbf{x}}}
\newcommand{\ff}{F^{\textbf{y},I}_A}
\newcommand{\y}{\textbf{y}}
\newcommand{\dd}{\hat{\Delta}}
\newcommand{\ha}[1]{\hat{\aaa}_A\left( #1 \right)}
\newcommand{\haa}{\hat{\aaa}_A}

\newcommand{\pg}{g^1_{i_1}\cdots g^d_{i_d}}
\newcommand{\altnor}[1]{| \! | \! | #1 | \! | \! |_{\pp}}
\newcommand{\robn}[1]{| \! | \! | #1 | \! | \! |}
\newcommand{\id}{{i_1,\ldots,i_d}}
\newcommand{\iid}{{i_1,\ldots,i_{d+1}}}

\newcommand{\gdd}{g^1_{i_1}\cdots g^d_{i_d}}
\newcommand{\rhp}{\hat{\rho}_A}
\newcommand{\mm}{\mathcal{M}}
\newcommand{\ccc}{\mathcal{C}}
\newcommand{\robnn}[1]{\left\langle #1 \right\rangle }
\newcommand{\rr}{\mathcal{R}}

\newtheorem{twie}{Theorem}[section]
\newtheorem{twr}[twie]{Theorem}
\newtheorem{lem}[twie]{Lemma}
\newtheorem{prep}[twie]{Proposition}
\newtheorem{prop}[twie]{Proposition}

\newtheorem{ex}[twie]{Example}
\newtheorem{cor}[twie]{Corollary}
\newtheorem{rem}[twie]{Remark}
\newtheorem{con}[twie]{Conjecture}

\begin{document}

\title{Moments of Gaussian chaoses in Banach spaces}
\author{Radosław Adamczak, Rafał Latała, Rafał Meller}
\date{}
\address{Institute of Mathematics, University of Warsaw, Banacha 2, 02--097 Warsaw, Poland.}
\email{r.adamczak@mimuw.edu.pl, r.latala@mimuw.edu.pl, r.meller@mimuw.edu.pl}
\thanks{ The authors were supported by National Science Centre, Poland grants 2015/18/E/ST1/00214 (R.A.),
2015/18/A/ST1/00553 (R.L.) and 2018/29/N/ST1/00454 (R.M.)}

\maketitle

\begin{abstract}
We derive moment and tail estimates for Gaussian chaoses  of arbitrary order with values in Banach spaces. We formulate a conjecture regarding two-sided estimates and show that it holds in a certain class of Banach spaces including $L_q$ spaces. As a
corollary we obtain two-sided bounds for moments of chaoses with values in $L_q$ spaces based on exponential random variables.
\end{abstract}

\section{Introduction}

Multivariate polynomials in Gaussian variables have been extensively studied at least since the work of Wiener in the 1930s. They have found numerous applications in the theory of stochastic integration and Malliavin calculus \cite{JansonGHS,NourdinPeccati,Nualart}, functional analysis \cite{Hytonen}, limit theory for $U$-statistics \cite{7} or long-range dependent processes \cite{Taqqu}, random graph theory \cite{JansonGHS}, and more recently computer science \cite{DeServedio,LarsenNelson,MOR,ODonnell}. While early results considered mostly polynomials with real coefficients, their vector-valued counterparts also appear naturally, e.g.,  in the context of stochastic integration in Banach spaces \cite{vNW}, in the study of weak limits of U-processes \cite{7}, as tools in characterization of various geometric properties of Banach spaces \cite{Hytonen,Pisier,PisierVol} or in analysis of empirical covariance operators \cite{A,Wahl}. Apart from applications, the theory of Gaussian polynomials has been studied for its rich intrinsic structure, with interesting interplay of analytic, probabilistic, algebraic and combinatorial phenomena, leading to many challenging problems. For a comprehensive presentation of diverse aspects of the theory we refer to the monographs \cite{7,Hytonen,JansonGHS,15}.

An important aspect of the study of Gaussian polynomials is the order of their tail decay and growth of moments. In the real valued case the first estimates concerning this question, related to the hypercontractivity of the Ornstein-Uhlenbeck semigroup,  were obtained by Nelson \cite{Nelson}. For homogeneous tetrahedral (i.e., affine in each variable) forms of arbitrary fixed degree two-sided estimates on the tails and moments were obtained in \cite{Latgaus} (in particular generalizing the well-known Hanson-Wright inequality for quadratic forms). In \cite{AW} it was shown that the results of \cite{Latgaus} in fact allow to obtain such estimates for all polynomials of degree bounded from above. Two sided estimates for polynomials with values in a Banach space have been obtained independently by  Borell \cite{5}, Ledoux \cite{14}, Arcones-Gin\'e \cite{3}. They are expressed in terms of suprema of certain empirical processes (see formula \eqref{a1} below), which in general may be difficult to estimate (even in the real valued case).

In a recent paper \cite{wek} we considered Gaussian quadratic forms with coefficients in a Banach space and obtained upper bounds on their tails and moments, expressed in terms of quantities which are easier to deal with. In the real valued case our estimates reduce to the Hanson-Wright inequality, and for a large class of Banach-spaces (related to Pisier's Gaussian property $\alpha$ and containing all type 2 spaces) they may be reversed. In particular for $L_q$ spaces with $q < \infty$ they yield two-sided estimates expressed in terms of deterministic quantities. In the present work we generalize these estimates to polynomials of arbitrary degree.

Before presenting our main theorems (which requires an introduction of a rather involved notation) let us describe the setting and discuss in more detail some of the results mentioned above.

To this aim consider a Banach space $(F,\lv \cdot \rv)$. A (homogeneous, tetrahedral) $F$-valued Gaussian chaos of order $d$ is a random variable defined as
\begin{equation}\label{eq:hom-tet-chaos}
S=\sum_{1\le i_1< i_2\ldots< i_d\le n} a_{\id} g_{i_1}\cdots g_{i_d},
\end{equation}
where $a_\id \in F$ and $g_1,\ldots,g_n$ are i.i.d.standard Gaussian variables. As explained above the goal of this paper is to derive estimates on moments (defined as $\lv S \rv_p:=(\E \lv S \rv^p)^{1/p}$) and tails of $S$, more precisely to establish upper bounds which for some classes of Banach paces, including $L_q$ spaces, can be reversed (up to constants depending only on $d$ and the Banach space, but not on $n$ or $a_\id$). We restrict here to random variables of the form \eqref{eq:hom-tet-chaos}, however it turns out that estimates on their moments will in fact allow to deduce moment and tail bounds for arbitrary polynomials in Gaussian random variables as well as for homogeneous tetrahedral polynomials in i.i.d. symmetric exponential random variables.

In the sequel  we will consider mainly decoupled chaoses
\[
S'=\sum_{i_1,\ldots, i_d} a_{i_1,\ldots,i_d} g^1_{i_1}\cdots g^d_{i_d},
\]
where $(g^k_i)_{i,k \in \N}$ are independent $\mathcal{N}(0,1)$ random variables -- under natural
symmetry assumptions, moments and tails of $S,S'$ are comparable up to constants depending only on $d$ (cf. \cite{pen,kwa}).

For $d=1$ and any $p\geq 1$ an easy application of Gaussian concentration (see, e.g. \cite{Ta_book}) and integration by parts gives
\begin{align}\label{eq:d=1}
\lv \sum_i a_i g_i \rv_p
&=\left( \E \lv \sum_i a_i g_i \rv^p \right)^{1/p}\sim \E \lv \sum_i a_i g_i \rv+\sup_{\varphi \in F^*, \lv \varphi \rv \leq 1} \lv \sum_i \varphi(a_i)g_i \rv_p\\
&\sim \E \lv \sum_i a_i g_i \rv+\sqrt{p} \sup_{x\in B^n_2} \lv \sum_i a_i x_i \rv,\nonumber
\end{align}
where $F^*$ is the dual space and $\sim$ stands for a comparison up to universal multiplicative constants.

An iteration of the above inequality yields for chaoses of order $2$,
\begin{align}\label{a12}
\lv \sum_{i,j}a_{ij}g_i g'_j \rv_p\sim&
\E \lv \sum_{i,j}a_{ij}g_i g'_j \rv+\sqrt{p} \E \sup_{x \in B^n_2} \left( \lv \sum_{i,j} a_{ij} g_i x_j \rv+\lv \sum_{i,j} a_{ij} x_i g_j \rv \right)\\
&+p\sup_{x,y \in B^n_2} \lv \sum_{i,j} a_{ij}x_i y_j \rv. \nonumber
\end{align}
For chaoses of higher order one gets an estimate
\begin{align}\label{a1}
\lv \sum_{i_1,\ldots,i_d} a_\id \gdd \rv_p \sim^d  \sum_{J\subset [d]} p^{d/2} \E \sup \lv \sum_{i_1,\ldots,i_d} a_\id \prod_{j \in J} x^j_{i_j} \prod_{j \in [d]\setminus J} g^j_{i_j} \rv,
\end{align}
where  the supremum is taken over $x^1,\ldots,x^n$ from the Euclidean unit sphere and $\sim^a$ stands for comparison up to constants depending only on the parameter $a$. To the best of our knowledge the above inequality was for the first time established in \cite{5} and subsequently reproved in various context by several authors \cite{3,14,15}.

The estimate \eqref{a1} gives precise dependence on $p$, but unfortunately is expressed in terms of expected  suprema of certain stochastic processes, which are hard to estimate. In many situations this precludes effective applications.
Let us note that even for $d=1$, the  estimate \eqref{eq:d=1} involves the expectation of a norm of a Gaussian random vector. Estimating such a quantity in general Banach spaces is a difficult task, which requires investigating the geometry of the unit ball of $F^\ast$ (as described by the celebrated majorizing measure theorem due to Fernique and Talagrand). Therefore, in general one cannot hope to get rid of certain expectations in the estimates for moments. Nevertheless, in certain classes of Banach spaces (such as, e.g., Hilbert spaces, or more generally type 2 spaces) expectations of Gaussian chaoses can be easily estimated.
The difficult part (also for $d=2$ and mentioned class of Banach spaces) is to estimate the terms in \eqref{a12} and \eqref{a1} which involve additional suprema over products of unit balls.  Even for $d=2$ and  a Hilbert space, the term
$\E \sup_{x \in B^n_2} \lv \sum_{i,j} a_{ij} g_i x_j \rv$ can be equivalently rewritten as the expected operator norm of a certain random matrix. Such quantities are known to be hard to estimate. Therefore, it is natural to seek inequalities which are expressed in terms of deterministic quantities and expectations of some $F$-valued polynomial chaoses, but do not involve expectations of additional suprema of such polynomials.
This was the motivation behind the article \cite{wek}, concerning the case $d=2$ and containing the following bound, valid for  $p \ge 1$ (\cite[Theorem 4]{wek}),
\begin{align*}
\lv \sum_{i,j} a_{i,j} g_i g'_j \rv_p
\leq & C \Bigg(\E \lv \sum_{i,j} a_{i,j} g_i g'_j \rv + \E \lv \sum_{i,j} a_{i,j}g_{ij} \rv \\
&+p^{1/2}\sup_{x \in B^n_2}\E   \lv \sum_{i,j} a_{ij} g_i x_j \rv
+p^{1/2}\sup_{x \in B^{n^2}_2} \lv \sum_{i,j} a_{ij} x_{ij}\rv \\&
+p \sup_{x,y \in B^n_2} \lv \sum_{i,j} a_{ij} x_i y_j \rv\Bigg).
\end{align*}

It can  be shown that in general this inequality cannot be reversed.  Hovewer, it turns out to be two-sided in a certain class of Banach spaces containing $L_q$ spaces (see Section \ref{sec:two-sided} below).
This observation gives rise  to the question of obtaining similar results for arbitrary $d$.
Building on ideas and techniques developed in \cite{Latgaus} we are able to give an answer to it (Theorem \ref{thm:gl}). Similarly as in \cite{wek} the heart of the problem is to estimate the expected  supremum of a certain Gaussian process indexed by a product set.

\medskip

The paper is organized as follows. In the next section we set up the notation and formulate the main results, in particular the pivotal bound for moments of homogeneous tetrahedral Gaussian chaoses in arbitrar Banach space (Theorem \ref{thm:gl}). We also present its consequences: tail and moment estimates for arbitrary Gaussian polynomials, two-sided bounds in special classes of Banach spaces, inequalities for tetrahedral homogeneous forms in i.i.d. symmetric exponential variables.
In Section \ref{sek3}, in Theorem \ref{maintheorem}, we formulate a key inequality for the supremum of a certain Gaussian processes and derive certain entropy bounds to be used in its proof, presented in Section \ref{sek4}. In Section \ref{sek5} we use Theorem \ref{maintheorem} to prove Theorem \ref{thm:gl} from which we deduce all the remaining claims of Section \ref{dwa2}. The Appendix contains certain basic facts concerning Gaussian processes and Gaussian polynomials used throughout the article.

\section{Notation and main results}\label{dwa2}
We write $[n]$ for the set $\{1,\ldots,n\}$. Throughout the article $C$ (resp. $C(\alpha)$) will denote an absolute constant (resp. a constant which may depend on $\alpha$) which may differ at each occurrence.  By $A$ we typically denote
a finite multi-indexed matrix $(a_\id)_{1\leq i_1,\ldots, i_d\leq n}$ of order $d$ with values in a normed space $(F,\lv  \cdot \rv)$. If $\ii = (i_1,\ldots,i_d) \in [n]^{d}$ and $I \subset [d]$, then we define $i_I:=(i_j)_{j \in I}$.

If $U$ is a finite set then $|U|$ stands for its cardinality and by $\pp(U)$ we denote the family of (unordered) partitions of $U$ into nonempty, pairwise disjoint sets. Note that if $U=
\emptyset$ then $\pp(U)$ consists only of the empty partition $\emptyset$.

 With a slight abuse of notation we write $(\pp,\pp ') \in \pp(U)$ if $\pp \cup \pp '\in \pp(U)$ and $\pp \cap \pp ' =\emptyset$.

Let $\pp=\{I_1,\ldots,I_k\},\ \pp '=\{J_1,\ldots,J_m \}$ be such that $(\pp, \pp ') \in \pp([d])$ . Then we define
\begin{align}
&\nor{A}:=\sup \left\{ \E \lv \sum_\id a_\id \prod_{r=1}^k x^r_{i_{I_r}}  \prod_{l =1 }^m g^{l}_{i_{J_l}}  \rv \ \Big{|}\ \forall_{r\leq k} \sum_{i_{I_r}} \left(  x^r_{i_{I_r}} \right)^2 \leq 1 \right\},\label{wspol} \\
&\altnor{A}:=\sup \left\{\E \lv \sum_\id a_\id \prod_{r=1}^k x^r_{i_{I_r}} \prod_{l \in [d]\setminus (\bigcup \pp)} g^l_{i_l} \rv \ \Big{|} \ \forall_{r\leq k} \sum_{i_{I_r}} \left(  x^r_{i_{I_r}} \right)^2 \leq 1\right\}.
\end{align}
We do not exclude the situation that $\pp '$ or $\pp$ is an empty partition. In the first case  $\nor{A}=\altnor{A}$ is defined in non-probabilistic terms. Another case when $\nor{A}=\altnor{A}$ is when $\mathcal{P}'$ consists of singletons only.

In particular for $d=3$ we have (note that to shorten the notation we suppress some brackets and write e.g.
$\robn{A}_{\{2\},\{3\}}$ and $\lv A \rv_{\{1\}  \sep \{2\},\{3\} }$,  instead of $\robn{A}_{\{\{2\},\{3\}\}}$
and $\lv A \rv_{\{\{1\}\}  \sep \{\{2\},\{3\}\} }$)
\[
\lv A \rv_{\emptyset \sep  \{1,2,3\}} =\robn{A}_{\{1,2,3\}}
=\sup_{\sum_{i,j,k} x^2_{ijk} \leq 1} \lv \sum a_{ijk}x_{ijk} \rv,
\]
\[
\lv A \rv_{\emptyset \sep \{1\},\{2,3\}}=\robn{A}_{\{1\},\{2,3\}}
= \sup_{ \sum_i x^2_i\leq 1, \sum_{jk}y_{jk}^2\leq 1 } \lv \sum_{ijk}a_{ijk}x_iy_{jk}\rv.
\]
\[
 \lv A \rv_{\emptyset \sep \{1\},\{2\},\{3\}}=\robn{A}_{\{1\},\{2\},\{3\}}
= \sup_{ \sum_i x^2_i\leq 1, \sum_{j}y_{j}^2\leq 1,\sum_k z_k^2\leq 1 } \lv \sum_{ijk}a_{ijk}x_iy_{j}z_k\rv.
\]
\[
\lv A \rv_{\{1,2\},\{3\} \sep \emptyset }  = \E \lv\sum_{i,j,k} a_{ijk} g_{ij} g'_k \rv,
\]
\[
\lv A \rv_{\{1\}  \sep \{2\},\{3\} }= \robn{A}_{\{2\},\{3\}}
=\sup_{\sum_{j} x^2_j \leq 1, \sum y^2_k \leq 1} \E \lv \sum_{i,j,k} a_{ijk} g_i x_j y_k \rv,
\]
\[
\|A\|_{\{1\},\{2\},\{3\} \sep\emptyset} = \robn{A}_{\emptyset} = \E \lv \sum_{ijk} a_{ijk} g_ig_j'g_k''\rv,
\]
\[
\|A\|_{\{1\},\{3\} \sep\{2\}} = \robn{A}_{\{2\}}= \sup_{\sum_j x^2_j \leq 1} \E \lv \sum_{i,j,k} a_{ijk} g_i x_j g'_k \rv.
\]

The main result is the following moment estimate of the variable $S'$.

\begin{twr}
\label{thm:gl}
Assume that  $A=(a_\id)_\id$ is a finite matrix with values in a normed space $(F,\lv \cdot \rv)$. Then for any $p\geq 1$,
\begin{align}
\label{ine:gl}
\frac{1}{C(d)} \sum_{J\subset [d]} \sum_{\pp \in \pp(J)}p^{{|\pp |}/{2}} \altnor{A}
&\leq \lv \sum_\id a_\id \pg \rv_p \\
\notag
&\phantom{aaaaaaaaaaaa}\leq C(d) \sum_{(\pp,\pp ') \in \pp([d])} p^{|\pp|/2} \nor{A}.
\end{align}
\end{twr}

The lower bound in \eqref{ine:gl} motivates the following conjecture (we leave it to the reader to verify that in general Banach spaces it is impossible to reverse the upper bound even for $d=2$).

\begin{con}
Under the assumption of Theorem \ref{thm:gl} we have
\begin{equation}
\lv \sum_\id a_\id \pg \rv_p \leq C(d)\sum_{J\subset [d]} \sum_{\pp \in \pp(J)}p^{{|\pp |}/{2}} \altnor{A}. \label{wz}
\end{equation}

\end{con}

\begin{ex}
In particular for $d=3$,  Theorem \ref{thm:gl} yields for symmetric matrices
\begin{align*}
\frac{1}{C}S_1\leq\lv \sum_{ijk} a_{ijk} g^1_i g^2_j g^3_k \rv_p \leq C(S_1+S_2),
\end{align*}
where (we recall) C is a numerical constant and
\begin{align*}
S_1&:=\E \lv \sum_{ijk} a_{ijk} g^1_i g^2_j g^3_k   \rv
\\
&+p^{1/2}\left(\sup_{\|x\|_2\leq 1} \E \lv \sum_{ijk} a_{ijk} g^1_{i}g^2_j x_k   \rv+\sup_{\|x\|_2\leq 1} \E \lv \sum_{ijk} a_{ijk} g_i x_{jk}  \rv
+\sup_{\|x\|_2\leq 1}  \lv \sum_{ijk} a_{ijk}  x_{ijk} \rv\right)
\\
&+p\left(\sup_{\|x\|_2,\|y\|_2\leq 1} \E \lv \sum_{ijk} a_{ijk} g_i x_j y_k  \rv+\sup_{\|x\|_2,\|y\|_2\leq 1}  \lv \sum_{ijk} a_{ijk}  x_{ij} y_k \rv \right)\\
&+p^{3/2} \sup_{\|x\|_2,\|y\|_2,\|z\|_2\leq 1}  \lv \sum_{ijk} a_{ijk} x_i y_j z_k  \rv
\\
S_2&:=\E \lv \sum_{ijk} a_{ijk} g_{ijk}  \rv+\E \lv \sum_{ijk} a_{ijk}  g^1_{ij} g^2_k \rv+p^{1/2} \sup_{\|x\|_2\leq 1} \E \lv \sum_{ijk} a_{ijk} g^1_{ij} x_k   \rv.
\end{align*}
\end{ex}

\begin{rem}\label{rem11}
Unfortunately we are able to show \eqref{wz} only for $d=2$ and with an additional factor $\ln p$ (cf. \cite{wek}). It is likely that by a  modification
 of our proof one can show \eqref{wz} for arbitrary $d$ with an additional factor $(\ln  p)^{C(d)}$.
\end{rem}

By a standard application of Chebyshev and Paley-Zygmund inequalities, Theorem \ref{thm:gl} can be expressed in terms of tails.

\begin{twr}\label{thm:og}
Under the assumptions of Theorem \ref{thm:gl} the following two inequalities hold.
 For any $t>C(d) \sum_{\pp ' \in \pp([d])} \lv A \rv_{\pp '|\emptyset}$,
\begin{align*}
\Pro \left(\lv \sum_\id a_\id \pg \rv \geq t  \right)
&\leq  2\exp \left(-\frac{1}{C(d)} \min_{\substack{(\pp,\pp ') \in \pp([d])\\ |\pp|>0}} \left(\frac{t}{\nor{A}} \right)^{2/|\pp|}  \right),
\end{align*}
and for any $t\geq 0$,
\begin{align*}
\Pro \left(\lv \sum_\id a_\id \pg \rv \geq \frac{1}{C(d)} \E \left\| \sum_{i_1,\ldots,i_d} a_{\id} \pg \right\| + t  \right)
\\
\geq \frac{1}{C(d)}\exp\left(-C(d) \min_{\emptyset\neq J\subset [d]} \min_{\pp \in \pp(J)} \left(\frac{t}{\altnor{A}} \right)^{2/|\pp|} \right) .
\end{align*}
\end{twr}

In  view of \eqref{a1} and \cite{Latgaus} it is clear that to prove
Theorem \ref{thm:gl} one needs to estimate suprema of some Gaussian processes.
The next statement is the key element of the proof of the upper bound in
\eqref{ine:gl}.
\begin{twr}\label{thm:proc}
Under the assumptions of Theorem \ref{thm:gl} we have for any $p\geq 1$
\begin{align}
\E \sup_{(x^2,\ldots,x^d) \in (B^n_2)^{d-1}} \lv \sum_{i_1,\ldots,i_d} a_{i_1,\ldots,i_d} g_{i_1} \prod_{k=2}^{d} x^k_{i_k} \rv \leq C(d) \sum_{(\pp,\pp ')\in \pp([d])} p^{\frac{|\pp|+1-d}{2}}\nor{A}. \label{nieprzet}
\end{align}
\end{twr}
 We postpone proofs of the above results till Section \ref{sec:proofs} and discuss now
some of their consequences.

\subsection{Two-sided estimates in special classes of Banach spaces}\label{sec:two-sided}
\paragraph{}

We start by introducing a class of normed spaces for which the estimate \eqref{ine:gl} is two-sided. To this end we restrict our attention to normed spaces $(F,\lv  \cdot \rv)$ which satisfy the following condition:
 there exists a constant $K=K(F)$ such that for any $n\in \N$ and any matrix $(b_{i,j})_{i,j\leq n}$ with values in $F$,
\begin{align}\label{warunek1}
\E \lv \sum_{i,j} b_{i,j} g_{i,j} \rv \leq K \E \lv \sum_{i,j} b_{i,j} g_i g'_j \rv.
\end{align}

This property appears in the literature under the name \emph{Gaussian property ($\alpha+$)} (see \cite{vNW}) and is closely related to Pisier's contraction property \cite{Pisier}. It has found applications, e.g., in the theory of stochastic integration in Banach spaces. We refer to \cite[Chapter 7]{Hytonen} for a thorough discussion and examples, mentioning only that \eqref{warunek1} holds for Banach spaces of type 2, and for Banach lattices \eqref{warunek1} is equivalent to finite cotype.

\begin{rem}
By considering $n=1$ it is easy to see that $K \ge \sqrt{\pi/2} > 1$.
\end{rem}

A simple inductive argument and \eqref{warunek1} yield that for any $d, n\in \N$ and any $F$-valued matrix $(b_{i_1,\ldots,i_d})_{i_1,\ldots,i_d\leq n}$,
\begin{align}\label{warunek2}
\E \lv \sum_\ii b_\ii g_\ii \rv \leq K^{d-1} \E \lv \sum_\ii b_\ii \pg \rv,
\end{align}
where we recall that  $\ii=(i_1,\ldots,i_d) \in [n]^d$.
It turns out that under the condition \eqref{warunek1} our bound   \eqref{ine:gl} is actually two-sided.

\begin{prep}\label{takie}
Assume that $(F,\lv  \cdot \rv)$ satisfies \eqref{warunek1} and $(\pp,\pp ') \in \pp([d])$. Then
$$
\nor{A}\leq K^{|\bigcup \pp '|-|\pp '|} \altnor{A}.
$$
\end{prep}

\begin{proof}
Let $\pp '=(J_1,\ldots,J_k),\ \pp=(I_1,\ldots,I_m)$. Then $|\bigcup \pp '|-|\pp '|=\sum_{l=1}^k(|J_l|-1)$.
The proof is by induction on $s:=| \{l \ : \ |J_l| \geq 2 \} |$. If $s=0$ the assertion follows by the definition
of $\altnor{A}$.
Assume that the statement holds for $s$ and $| \{l \ : \ |J_l| \geq 2 \} |=s+1$. Without loss of generality $|J_1|\geq 2$. Combining Fubini's Theorem with \eqref{warunek2} we obtain
\begin{align*}
&\nor{A}=\sup \left\{\E^{(G^2,\ldots,G^m)} \E^{G^1}\lv\sum_\ii a_\ii \prod_{r=1}^m x^r_{i_{I_r}} g^1_{i_{J_1}}\prod_{r=2}^k g^r_{i_{J_r}}\rv \ \Big{|} \ \forall_{r\leq m} \sum_{i_{ I_r}}  \left(x^r_{i_{I_r}}\right)^2 \leq 1 \right\} \\
&\leq K^{|J_1|-1} \sup \left\{\E^{(G^2,\ldots,G^m)} \E^{G'}\lv\sum_\ii a_\ii \prod_{r=1}^m x^r_{i_{I_r}} \prod_{j \in J_1} (g')^j_{i_j}\prod_{r=2}^k g^r_{i_{J_r}}\rv \ \Big{|} \ \forall_{r\leq m} \sum_{i_{ I_r}}  \left(x^r_{i_{I_r}}\right)^2 \leq 1 \right\} \\
&\leq K^{|\bigcup \pp '|-|\pp '|} \altnor{A},
\end{align*}
where $G^l=(g^l_{i_{I_l}})_{i_{I_l}}$, $G'=((g')^j_{i_j})_{j \in J_1,i_{J_1}}$ and in the last inequality we used the induction assumption.
\end{proof}
The following corollary is an obvious consequence of Proposition \ref{takie} and Theorems \ref{thm:gl}, \ref{thm:og}.

\begin{cor}\label{cor:special-spaces}
For any normed space $(F,\lv \cdot \rv)$ satisfying \eqref{warunek1} we have for $p\geq 1$,
\begin{equation}
\lv \sum_\id a_\id \pg \rv_p\leq C(d) K^{d-1} \sum_{T\subset [d]}\sum_{\pp \in \pp(T)} p^{|\pp|/2} \altnor{A},\nonumber
\end{equation}
and for   $t>C(d) K^{d-1}\E \lv \sum_\ii a_\ii \pg \rv$,
\begin{equation}
\Pro \left(\lv \sum_\id a_\id \pg \rv \geq t  \right)
\leq 2\exp\left(-\frac{1}{C(d)}K^{2-2d} \min_{\emptyset\neq J\subset [d]} \min_{\pp \in \pp(J)} \left(\frac{t}{\altnor{A}} \right)^{2/|\pp|} \right). \nonumber
\end{equation}
\end{cor}

Thanks to infinite divisibility of Gaussian variables, the above corollary can be in fact generalized to arbitrary polynomials in Gaussian variables, as stated in the following theorem.

\begin{twr}\label{thm:general-polynomials}Let $F$ be a Banach space. If $G$ is a standard Gaussian vector in $\R^n$ and $f\colon \R^n \to F$ is a polynomial of degree $D$, then for all $p \ge 2$,
\begin{align}\label{eq:general-poly-1}
&\|f(G) - \E f(G)\|_p\\
&\ge \frac{1}{C(D)} \Big(\E \|f(G) - \E f(G)\|+ \sum_{1\le d\le D}\sum_{\emptyset \neq T\subset [d]} \sum_{\pp \in \pp(T)} p^{\frac{|\pp|}{2}} \altnor{\E \D^df(G)}\Big)\nonumber
\end{align}
and for all $t > 0$,
\begin{align}\label{eq:general-poly-2}
\Pro\Big(\|f(G) - \E f(G)\| \ge \frac{1}{C(D)} (\E \|f(G) - \E f(G)\| + t)\Big) \ge \frac{1}{C(D)}\exp\Big(-C(D) \eta_f(t)\Big),
\end{align}
where
\begin{align*}
\eta_f(t) = \min_{1\le d\le D}\min_{\emptyset \neq T\subset [d]} \min_{\pp \in \pp(T)} \Big(\frac{t}{\altnor{\E \D^d f(G)}}\Big)^{2/|\pp|}.
\end{align*}
Moreover, if $F$ satisfies \eqref{warunek1}, then for all $p \ge 1$,
\begin{align}\label{eq:general-poly-3}
&\|f(G) - \E f(G)\|_p\\
&\le C(D)K^{D-1} \Big(\E \|f(G) - \E f(G)\|+ \sum_{1\le d\le D}\sum_{\emptyset \neq T\subset [d]} \sum_{\pp \in \pp(T)} p^{\frac{|\pp|}{2}} \altnor{\E \D^df(G)}\Big)\nonumber
\end{align}
and for all $t \ge C(D) K^{D-1}\E \|f(G) - \E f(G)\|$,
\begin{align}\label{eq:general-poly-4}
\Pro\Big(\|f(G) - \E f(G)\| \ge t)\Big) \le 2\exp\Big(-C(D)^{-1} K^{2-2D}\eta_f(t)\Big).
\end{align}
\end{twr}

The above theorem is an easy consequence of results for homogeneous decoupled chaoses and the following proposition, the proof of which (as well as the proof of the theorem) will be presented in Section \ref{sec:proofs}.

\begin{prep}\label{prop:splitting-derivatives} Let $F$ be a Banach space, $G$ a standard Gaussian vector in $\R^n$ and $f\colon \R^n \to F$ be a polynomial of degree $D$. Then for $p\ge 1$,
\begin{displaymath}
  \|f(G) - \E f(G)\|_p \sim^D \sum_{d=1}^D \Big \|\sum_{i_1,\ldots,i_d=1}^n a^d_{i_1,\ldots,i_d} g_{i_1}^1\cdots g_{i_d}^d\Big\|_p,
\end{displaymath}
where the $d$-indexed $F$-valued matrices $A_d = (a_{i_1,\ldots,i_d})_{i_1,\ldots,i_d\le n}$ are defined as $A_d = \E \nabla^d f(G)$.
\end{prep}

\subsection{$L_q$ spaces}
It turns out that    $L_q$ spaces satisfy  \eqref{warunek1} and as a result
upper and lower bounds in \eqref{ine:gl} are comparable. Moreover, as is shown in Lemma \ref{det} below,  in this case one may express all the parameters without any expectations. For the sake of brevity, we will focus on moment estimates, clearly tail bounds follow from them by standard arguments (cf. the proof of Theorem \ref{thm:og}).

\begin{prop}
\label{ft:Lq}
For $q \ge 1$ the space $L_q(X,\mu)$ satisfies \eqref{warunek1} with $K=C\sqrt{q}$.
\end{prop}

\begin{proof} From \cite[Theorem 7.1.20]{Hytonen} it follows that if $F$ is of type 2 with constant $T$, then it satisfies \eqref{warunek1} with $K = T$, while it is well known that the type 2 constant of $L_q(X,\mu)$ is of order
$\sqrt{q}$.
\end{proof}



For a multi-indexed matrix $A$ of order $d$ with values in $L_q(X,\mu)$ and $J\subset [d]$, $\pp=(I_1,\ldots,I_k) \in \pp([J])$ we define
$$
\altnor{A}^{L_q}=\sup \left\{\lv  \sqrt{\sum_{i_{[d]\setminus J}} \left(\sum_{i_J} a_\ii \prod_{r=1}^k x^r_{i_{I_r}} \right)^2} \rv_{L_q}  \ \Big{|} \ \forall_{r\leq k} \sum_{i_{I_r}} \left(  x^r_{i_{I_r}} \right)^2 \leq 1\right\} .
$$

For $J = [d]$ and $\pp \in \pp([d])$ we obviously have $\altnor{A}^{L_q} = \altnor{A}$. The following lemma asserts that for general $J$ the corresponding two norms are comparable.
\begin{lem}\label{det}
For any $J \subsetneq [d]$, $\pp=(I_1,\ldots,I_k) \in \pp(J)$ and any multi-indexed matrix $A$ of order $d$ with values in $L_q(X,\mu)$ we have
\begin{align*}
C(d)^{-1}q^{\frac{1-d+|J|}{2}} \altnor{A}^{L_q}\leq \altnor{A}\leq C(d)q^{\frac{d-|J|}{2}}\altnor{A}^{L_q}.
\end{align*}
\end{lem}

\begin{proof}
By Jensen's inequality and  Corollary \ref{lgaus} we get
\begin{align*}
\altnor{A}&\leq \sup \left\{ \left(\int_X \E \left|\sum_\ii a_\ii(x) \prod_{j \in [d]\setminus J} g^j_{i_j} \prod_{r=1}^k x^r_{i_{I_r}} \right|^q d\mu(x) \right)^{1/q} \ \Big{|} \ \forall_{r\leq k} \sum_{i_{I_r}} \left(  x^r_{i_{I_r}} \right)^2 \leq 1 \right\} \\
&\leq C(d) q^{\frac{d-|J|}{2}} \sup \left\{ \lv \sqrt{\sum_{i_{[d]\setminus J}} \left(\sum_{i_J} a_\ii  \prod_{r=1}^k x^r_{i_{I_r}}  \right)^2}\rv_{L_q} \ \Big{|} \ \forall_{r\leq k} \sum_{i_{I_r}} \left(  x^r_{i_{I_r}} \right)^2 \leq 1 \right\}.
\end{align*}
 On the other hand Theorem \ref{gaushiper} (applied with $p=1$) and Corollary \ref{lgaus}  yield
\begin{align*}
\altnor{A}
&\geq \frac{q^{\frac{|J|-d}{2}}}{C(d)} \sup \left\{ \left(\E \lv\sum_\ii a_\ii(x) \prod_{j \in [d]\setminus J} g^j_{i_j} \prod_{r=1}^k x^r_{i_{I_r}} \rv_{L_q}^q \right)^{1/q} \ \Big{|} \ \forall_{r\leq k} \sum_{i_{I_r}} \left(  x^r_{i_{I_r}} \right)^2 \leq 1 \right\}
\\
&\geq  \frac{q^{\frac{1-d+|J|}{2}}}{C(d)}\sup \left\{\lv  \sqrt{\sum_{i_{[d]\setminus J}} \left(\sum_{i_J} a_\ii \prod_{r=1}^k  x^r_{i_{I_r}}\right)^2} \rv_{L_q}  \ \Big{|} \ \forall_{r\leq k} \sum_{i_{I_r}} \left(  x^r_{i_{I_r}} \right)^2 \leq 1\right\}.
\end{align*}
\end{proof}

\begin{twr}\label{Lq}
Let $q \ge 1$ and let $A=(a_\id)_\id$ be a matrix with values in $ L_q(X,\mu) $. Then for any $p\geq 1$ we have
\begin{align*}
\frac{1}{C(d)} q^{\frac{1-d}{2}}\sum_{J\subset [d]} \sum_{\pp\in \pp([J])} p^{\frac{|\pp|}{2}}\altnor{A}^{L_q}
&\leq
\lv \sum_\id a_\id \pg \rv_p
\\
&\phantom{aaaaaaaa}
\leq C(d)q^{ d-\frac{1}{2}} \sum_{J\subset [d]} \sum_{\pp \in \pp([J])} p^{\frac{|\pp|}{2}} \altnor{A}^{L_q}.
\end{align*}

\end{twr}
\begin{proof}
This is an obvious consequence of Theorem \ref{thm:gl}, Corollary \ref{cor:special-spaces}, Proposition \ref{ft:Lq} and Lemma \ref{det}.
\end{proof}

Using Proposition \ref{prop:splitting-derivatives} we can extend the above result to general polynomials.

\begin{twr}
Let $G$ be a standard Gaussian vector in $\R^n$ and $f\colon \R^n \to L_q(X,\mu)$ ($q \ge 1$) a polynomial of degree $D$. Then for $p \ge 1$, we have
\begin{align*}
\frac{1}{C(D)} \sum_{d=1}^D q^{\frac{1-d}{2}} &\sum_{J\subset [d]} \sum_{\pp \in \pp([J])} p^{\frac{|\pp|}{2}} \altnor{\E \D^d f(G)}^{L_q}
\leq
\|f(G) - \E f(G)\|_p
\\
 &\le C(D)\sum_{d=1}^D q^{d - \frac{1}{2}} \sum_{J\subset [d]} \sum_{\pp \in \pp([J])} p^{\frac{|\pp|}{2}} \altnor{\E \D^d f(G)}^{L_q}.
\end{align*}
\end{twr}

\begin{ex}
Consider a general polynomial of degree $3$, i.e.,
\begin{displaymath}
  f(G) = \sum_{i,j,k=1}^n a_{ijk}g_ig_jg_k + \sum_{i,j=1}^n b_{ij}g_i g_j + \sum_{i=1}^n c_i g_i + d,
\end{displaymath}
where the coefficients $a_{ijk}, b_{ij}, c_{i}, d$ take values in a Banach space and the matrices $(a_{ijk})_{ijk}$, $(b_{ij})_{ij}$ are symmetric.
Then one checks that
\begin{align*}
\E \nabla f(G) &= \Big(c_i + 3\sum_{j=1}^n a_{ijj}\Big)_{i=1}^n,\\
\E \nabla^2 f(G) & = 2(b_{ij})_{i,j=1}^n,\\
\E \nabla^3 f(G) & = \nabla^3 f(G) = 6(a_{ijk})_{i,j,k=1}^n.
\end{align*}
\end{ex}

\subsection{Exponential variables}
Theorem \ref{Lq} together with Lemma \ref{lem95} allows us to obtain inequalities for chaoses based on i.i.d standard symmetric exponential random variables (i.e. variables with density $2^{-1} \exp(-|t|)$) which are denoted by  $(E^i_j)_{i,j \in \N}$ below. Similarly as in the previous Section we concentrate only on the moment estimates.

\begin{prop}
\label{wyk}
Let $A=(a_\id)_\id$ be a matrix with values in $ L_q(X,\mu) $. Then for any $p\geq 1, \ q\geq 2$ we have
\begin{align*}
 \lv \sum_\ii a_\ii \prod_{k=1}^d E_{i_k}^{k} \rv_p \sim^{d,q}
 \sum_{I\subset [d]}\sum_{J\subset [d]\setminus I}
\sum_{\pp\in \pp([d]\setminus (I\cup J))} p^{|I|+|\pp|/2} \max_{i_{I}} \robn{(a_{\id})_{i_{I^c}}}^{L_q}_{\pp}.
\end{align*}
One can take $C^{-1}(d) q^{1/2-d} $ in the lower bound and $C(d) q^{2d-  1/2}$ in the upper bound.
\end{prop}

\begin{ex}
If $d=2$ then Proposition \ref{wyk} reads for a symmetric matrix $A=(a_{ij})_{ij}$ as
\begin{align*}
\lv \sum_{ij} a_{ij} E^1_i E^2_j \rv_p
\sim^q& p^2\max_{i,j} \lv a_{ij}\rv_{L_q}
+p^{3/2} \max_i \sup_{x \in B^n_2} \lv \sum_j a_{ij} x_j \rv_{L_q}
\\
&+p\left( \max_{x,y \in B^n_2} \lv \sum_{ij} a_{ij} x_i y_j \rv_{L_q}
+\max_i \lv \sqrt{\sum_j a^2_{ij}} \rv_{L_q}\right)
\\
&+p^{1/2}\left(
\sup_{x \in B^n_2} \lv \sqrt{\sum_i \left(\sum_j a_{ij} x_j \right)^2 } \rv_{L_q}
+\sup_{x \in B^{n^2}_2} \lv \sum_j a_{ij} x_{ij} \rv_{L_q}\right)
\\
&+\lv \sqrt{\sum_{ij} a^2_{ij}} \rv_{L_q}.
\end{align*}
\end{ex}

The proof of Proposition \ref{wyk} is postponed till Section \ref{sec:proofs}.

\section{Reformulation of Theorem \ref{thm:proc} and entropy estimates} \label{sek3}

Let us rewrite Theorem \ref{thm:proc} in a different language. We may assume that $F=\R^m$ for some
finite $m$ and $a_{i_1,\ldots,i_d}=(a_{i_1,\ldots,i_d,i_{d+1}})_{i_{d+1}\leq m}$. For this reason from now on the multi-index $\ii$ will take values in $[n]^d\times [m]$ and all summations over $\ii$ should be understood as summations over this set. Accordingly, the matrix $A$ will be treated as a $(d+1)$-indexed matrix with real coefficients. Let $T=B_{F^*}$ be the unit ball in the dual space $F^*$ (where duality is realised on $\R^m$ through the standard inner product).
In the sequel we will therefore assume that $T$ is a fixed nonempty symmetric bounded subset of $\R^m$.

In this setup we have
\begin{align}
\E \sup_{ (x^2,\ldots,x^d) \in (B^n_2)^{d-1}} \lv \sum_\ii a_\ii g_{i_1} \prod_{k=2}^{d} x^k_{i_k} \rv
=\E \sup_{ (x^2,\ldots,x^d) \in (B^n_2)^{d-1}} \sup_{t \in T} \sum_\ii a_\ii g_{i_1} \prod_{k=2}^{d} x^k_{i_k} t_{i_{d+1}}\nonumber
\\
\nor{A}=\sup \left\{ \E  \sup_{t \in T} \sum_\ii \ai \prod_{r=1}^k x^r_{i_{I_r}}  \prod_{s =1 }^l g^{s}_{i_{J_s}}  t_{i_{d+1}} \ \Big{|}\ \forall_{j=1\ldots,k} \sum_{i_{I_j}} \left(x^{(j)}_{i_{I_j}}\right)^2=1 \right\},
\label{eq:new-norm-definition}
\end{align}
where $\pp=(I_1,\ldots,I_k),\pp '=(J_1,\ldots,J_l)$, $(\pp',\pp) \in \pp([d])$.

To make the notation more compact we define
$$
s_k(A)= \sum_{\stackrel{(\pp,\pp ') \in \pp([d]) }{|\pp|=k}} \nor{A}.
$$
As we will see in Section \ref{sec:proofs}, to prove Theorem \ref{thm:proc} it suffices to show the following.

\begin{twr}\label{maintheorem}
For any $p\geq 1$  we have
\begin{align}
\E \sup_{(x^2,\ldots,x^d,t) \in (B^n_2)^{d-1}\times T} \sum_\ii a_\ii g_{i_1} \prod_{k=2}^{d} x^k_{i_k} t_{i_{d+1}} \leq C(d) \sum_{k=0}^{d} p^{\frac{k+1-d}{2}}s_k(A) . \label{ala}
\end{align}
\end{twr}

To estimate the supremum of a centered Gaussian process $(G_{v})_{v\in V}$ one needs to study the distance
on $V$ given by $d(v,v'):=(\E |G_v-G_{v'}|^2)^{1/2}$ (cf. \cite{Ta_book}). In the case of the Gaussian process from \eqref{ala} this distance is defined on $(B_2^n)^{d-1}\times T\subset \R^{n(d-1)}\times \R^m$
by the formula
\begin{align}
\label{eq:distance-definition}
&\rho_A((x^2,\ldots,x^d,t),(y^2,\ldots,y^d,t'))\\
&:=\left(\sum_{i_1} \left(\sum_{i_2,\ldots,i_{d+1}} a_\iid \left(\prod_{k=2}^d y^k_{i_k}t'_{i_{d+1}}-\prod_{k=2}^d x^k_{i_k}t_{i_{d+1}} \right) \right)^{2} \right)^{1/2}\nonumber\\
&=\al{\left(\bt_{k=2}^{d} x^k\right) \ten t - \left(\bt_{k=2}^{d} y^k\right) \ten t'},\nonumber
\end{align}
where $\left(\bt_{k=2}^{d} x^k \right) \ten t=(x^2_{i_2}\cdots x^d_{i_d}t_{i_{d+1}})_{i_2,\ldots,i_{d+1}}\in \R^{n^{d-1}m}$ and $\alpha_A$ is a norm defined on $(\R^n)^{\otimes (d-1)} \otimes \R^m \simeq \R^{n^{d-1}m}$ given by
\begin{align}
\label{eq:norm-definition}
\al{\textbf{x}}&:=\sqrt{\sum_{i_1}\left(\sum_{i_{[d+1]\setminus \{1\} }}a_\ii \textbf{x}_{i_{[d+1]\setminus \{1\}}} \right)^2}.
\end{align}

We will now provide estimates for the entropy numbers $N(U, \rho_A, \eps)$ for $\eps > 0$ and $U\subset (B^n_2)^{d-1}\times T$ (recall that $N(S,\rho,\eps)$ is the minimal number of closed balls with diameter $\eps$ in metric $\rho$ that cover the set $S$). To this end let us introduce some new notation. From now on $G_n = (g_1,\ldots,g_n)$ and $G^i_n=(g^i_1,\ldots,g^i_n)$ stand for independent standard Gaussian vectors in $\R^n$. For $s>0$, $U=\{ (x^2,\ldots,x^{d},t) \in U \} \subset \left(\R^{n}\right)^{d-1} \times T$  we set
\begin{align}\label{loc}
W^U_d(\alpha_A,s):=\sum_{k=1}^{d-1} s^k \sum_{I\subset \{2,\ldots,d \}: |I|=k} W^U_I(\alpha_A),
\end{align}
where
$$W^U_I(\alpha_A):=\sup_{(x^2,\ldots,x^{d+1},t) \in U} \E \al{\left(\los{2}{d}\right) \ten t}. $$

We define a norm $\beta_A$ on   $ (\R^n)^{\otimes (d-1)}\simeq \R^{n^{d-1}} $ by (recall that we assume symmetry of the set $T$)
\begin{align}\label{eq:definition-beta}
\bb{\textbf{y}}:=
\E \sup_{t \in T} \sum_\ii a_\ii g_{i_1} \textbf{y}_{i_{[d]\setminus \{1\}}} t_{i_{d+1}}
=\E \sup_{t \in T} \left| \sum_\ii a_\ii g_{i_1} \textbf{y}_{i_{[d]\setminus \{1\}}} t_{i_{d+1}} \right|.
\end{align}

Following \eqref{loc} we denote
\begin{align}\label{eq:definition-V}
V^U_d(\be_A,s):=\sum_{k=0}^{d-1} s^{k+1} \sum_{I\subset \{2,\ldots,d \}: |I|=k} V^U_I(\be_A),
\end{align}
where
$$
V^U_I(\be_A):=\sup_{(x^2,\ldots,x^d,t) \in U} \E \bb{\los{2}{d}}.
$$
Let us note that $V^U_I(\be_A)$ depends on the set $U$ only through its projection on the first $d-1$ coordinates.

We have
\begin{equation}\label{zaw}
V^U_d(\beta_A,s)\geq s \cdot V^U_{\emptyset}(\be_A)=s \cdot \sup_{(x^2,\ldots,x^d,t) \in U}  \bb{\bt_{k=2}^d x^k}.
\end{equation}
Observe that by the classical Sudakov minoration (see Theorem \ref{norsud}), for any $(x^k) \in \R^n$, $k=2,\ldots,d$ there exists $T_{\bt x^k,\eps} \subset T$ such that $|T_{\bt x^k,\eps}|\leq \exp(C \eps^{-2})$ and
$$
\forall_{t \in T} \exists_{t' \in T_{\bt x^k,\eps}} \al{\bt_{k=2}^d x^k \ten (t-t')} \leq \eps \bb{\bt_{k=2}^d x^k }.
$$
We define a measure $\mu^d_{\eps,T}$ on $\R^{(d-1)n} \times T$ by the formula
$$
\mu^d_{\eps,T}(C):=\int_{\R^{(d-1)n}} \sum_{t \in T_{\bt x^k,\eps}} \1_{C}\big((x^2,\ldots,x^d,t)\big) d\gamma_{(d-1)n,\eps}((x^k)_{k=2,\ldots,d}),
$$
where $\gamma_{n,t}$ is the distribution of $tG_n=t(g_1,\ldots,g_n)$. Clearly,
\begin{equation}
\label{ogrmiary}
\mu^d_{\eps,T}((\R^{d-1})^n \times T) \leq e^{C \eps^{-2}}.
\end{equation}

To bound $N(U, \rho_A, \eps)$ for $\eps > 0$ and $U\subset (B^n_2)^{d-1}\times T$ we need two lemmas.

\begin{lem}\cite[Lemma 2]{Latgaus}
\label{entgausog}
For any $\textbf{x}=(x^1,\ldots,x^d) \in (B^n_2)^d$, norm $\alpha '$ on $\R^{n^d}$ and $\eps>0$ we have
$$
\gamma_{dn,\eps}\left( B_{\alpha '}(\textbf{x},r(4\eps,\alpha')) \right) \geq 2^{-d} \exp(-d\eps^{-2}/2),
$$
where
$$
B_{\alpha '}(\textbf{x},r(\eps,\alpha'))=\left\{\textbf{y}=(y^1,\ldots,y^d) \in (\R^n)^{d} \ \mid \
\alpha ' \left(\bigotimes_{k=1}^d x^k-\bigotimes_{k=1}^d y^k\right) \leq r(\eps,\alpha ') \right\},
$$
and
$$
r(\eps,\alpha ')=\sum_{k=1}^d \eps^k \sum_{I \subset [d]:\ |I|=k}
\E \alpha '\left(\bigotimes_{k=1}^d \left(x^k (1-\1_{k \in I})+G^k \1_{k \in I} \right) \right).
$$
\end{lem}

\begin{lem}\label{miarkul}
For any $(\xx,t)=(x^2,\ldots,x^d,t) \in \left( B^n_2 \right)^{d-1} \times T$ and $\eps>0$ we have
$$
\mu^d_{\eps,T}\left(B\left( (\xx,t),\rho_A, W^{\{(\xx,t)\}}_d(\alpha_A,8\eps)+V^{\{(\xx,t)\}}_d(\be_A,8\eps) \right) \right)\geq c^d \exp\left(- C(d)\eps^{ -2}\right).
$$
\end{lem}

\begin{proof}
Fix  $(\xx,t)\in \left( B^n_2 \right)^{d-1} \times T$, $\eps>0$ and consider
\begin{align*}
U=\Bigg\{(y^2,\ldots,y^d) \in \R^{(d-1)n} \colon\  &\al{\left(\bt_{k=2}^d x^k-\bt_{k=2}^d y^k \right)\ten t } + \varepsilon \bb{\bt_{k=2}^d x^k-\bg_{k=2}^d y^k}
\\
 & \leq W^{\{(\xx,t)\}}_d(\alpha_A,4\eps)+V^{\{(\xx,t)\}}_d(\be_A,4\eps)  \Bigg\}.
\end{align*}
For any $(y^2,\ldots,y^d) \in U$ there exists $t' \in T_{\bt y^k,\eps}$ such that $ \al{\bt_{k=2}^d y^k \ten (t-t')} \leq \eps \bb{\bt_{k=2}^d y^k }$. By the triangle inequality,
\begin{align*}
&\al{\bg_{k=2}^d x^k \ten t - \bg_{k=2}^d y^k \ten t'}
\\
&\phantom{aaaaa}
\leq \al{\left( \bg_{k=2}^d x^k -\bg_{k=2}^d y^k \right) \ten t}+\al{\bg_{k=2}^d y^k \ten (t-t')}
\\
&\phantom{aaaaa}
\leq\al{\left( \bg_{k=2}^d x^k -\bg_{k=2}^d y^k \right) \ten t} +\eps \bb{\bg_{k=2}^d x^k-\bt_{k=2}^d y^k }+\eps \bb{\bg_{k=2}^d x^k }
\\
&\phantom{aaaaa}
\leq W^{\{(\xx,t)\}}_d(\alpha_A,4\eps)+2V^{\{(\xx,t)\}}_d(\be_A,4\eps)
\leq  W^{\{(\xx,t)\}}_d(\alpha_A,8\eps)+V^{\{(\xx,t)\}}_d(\be_A,8\eps),
\end{align*}
where in the third inequality we used \eqref{zaw}. Thus,
\begin{multline*}
\mu^d_{\eps,T}\left(B\left( (\xx,t),\rho_A, W^{\{(\xx,t)\}}_d(\alpha_A,8\eps)
+V^{\{(\xx,t)\}}_d(\be_A,8\eps) \right) \right)
\\
\geq \gamma_{(d-1)n,\eps}(U)\geq c^d \exp(-C(d) \eps^{ -2}),
\end{multline*}
where the last inequality follows by Lemma \ref{entgausog}  applied to the norm $\alpha_A(\cdot \otimes t) + \varepsilon\beta_A( \cdot)$.

\end{proof}

\begin{cor}\label{gorsz}
For any $\eps,\delta>0$ and $U\subset (B^n_2)^{d-1}\times T$ we have
\begin{align}
N\left(U,\rho_A,W^U_d(\alpha_A,\eps)+V^U_d(\be_A,\eps)\right) \leq \exp(C(d) \eps^{-2}) \label{k1}
\end{align}
and
\begin{align}
\label{k2}
&\sqrt{\log N(U,\rho_A,\delta)}
\\
\notag
&\leq C(d) \left(\sum_{k=1}^{d-1} \left(\sum_{\stackrel{I\subset \{2,\ldots,d\}}{|I|=k}}  W^U_I(\alpha_A)\right)^{\frac{1}{k}}\delta^{-\frac{1}{k}}+\sum_{k=0}^{d-1} \left(\sum_{\stackrel{I \subset \{2,\ldots,d\}}{|I|=k}}V^U_I(\beta_A) \right)^{\frac{1}{k+1}} \delta^{-\frac{1}{k+1}}\right).
\end{align}
\end{cor}

\begin{proof}
It suffices to show \eqref{k1}, since it easily implies \eqref{k2}. Consider first $\eps\leq 8$. Obviously,
$W^U_d(\alpha_A,\eps)+V^U_d(\be_A,\eps)\geq \sup_{(\xx,t) \in U } (W^{\{(\xx,t)\}}_d(\alpha_A,\eps)
+V^{\{(\xx,t)\}}_d(\be_A,\eps))$. Therefore, by Lemma \ref{miarkul} (applied with $\eps/16$)
we have for any $(\xx,t)\in U$,
\begin{equation}
\label{gorsz1}
\mu^d_{\eps,T}\left(B\left( (\xx,t),\rho_A, W^U_d(\alpha_A,\eps/2)+V^U_d(\be_A,\eps/2) \right) \right)
\geq C(d)^{-1} \exp\left(-C(d)\eps^{-2}\right).
\end{equation}
Suppose that there exist $(\xx_1,t_1),\ldots,(\xx_N,t_N) \in U$ such that
$\rho_A ( (\xx_i,t_i),(\xx_j,t_j))> W^U_d(\alpha_A,\eps)+V^U_d(\be_A,\eps) \geq 2W^U_d(\alpha_A,\eps/2)+2V^U_d(\be_A,\eps/2)$ for $i \neq j$.
Then the sets $B\left( (\xx_i,t_i),\rho_A, W^U_d(\alpha_A,\eps/2)+V^U_d(\be_A,\eps/2) \right) $ are disjoint, so by
\eqref{ogrmiary} and \eqref{gorsz1}, we obtain $N \leq  C(d)\exp(C(d) \eps^{-2}) \leq \exp(C(d) \eps^{-2}) $.

If $\eps\geq 8$ then \eqref{zaw} gives
\begin{align*}
W^U_d(\alpha_A,\eps)+V^U_d(\be_A,\eps)
&\geq 8 \sup_{(x^2,\ldots,x^d,t) \in U } \E \left| \sum_\ii a_\ii g_{i_1} \prod_{k=2}^d x^k_{i_k} t_{i_{d+1}} \right|
\\
&=\sqrt{\frac{128}{\pi}} \sup_{(x^2,\ldots,x^d,t) \in U } \left(\sum_{i_1} \left(\sum_{i_2,\ldots,i_{d+1}}  a_\ii \prod_{k=2}^d x^k_{i_k} t_{i_{d+1}}  \right)^2 \right)^{1/2}
\\
&\geq \mathrm{diam}\left(U, \rho_A \right).
\end{align*}
So, $N(U,\rho_A,W^U_d(\alpha_A,\eps)+V^U_d(\be_A,\eps))=1 \leq \exp( \eps^{-2})$.
\end{proof}

\begin{rem}
\newcommand{\U}{{(B^n_2)^{d-1}\times T}}
The classical Dudley's bound on suprema of Gaussian processes (see e.g., \cite[
Corollary 5.1.6]{7}) gives
\begin{align*}
\E \sup_{(x^2,\ldots,x^d,t) \in (B^n_2)^{d-1}\times T} \sum_\ii a_\ii g_{i_1} \prod_{k=2}^{d} x^k_{i_k} t_{i_{d+1}}\leq C \int_{0}^{\Delta} \sqrt{\log N(\U,\rho_A,\delta)} d \delta,
\end{align*}
where $\Delta$ is equal to the diameter of the set $(B^n_2)^{d-1}\times T$ in the metric $\rho_A$.
Unfortunately the entropy bound derived in Corollary \ref{gorsz} involves a nonintegrable term $\delta^{-1}$.
The remaining part of the proof of Theorem \ref{maintheorem} is devoted to
improve on Dudley's bound.
\end{rem}

For $\textbf{x},\textbf{y}\in (\R^{n})^{d-1}$ we define a norm $\hat{\alpha}_A$ on $ (\R^{n})^{d-1}=\R^{(d-1)n}$
by the formula
\begin{multline*}
\hat{\alpha}_A((x^2,\ldots,x^d))
:=\sum_{j=2}^d \sum_{\substack{(\pp,\pp ') \in \pp([d]\setminus \{j\}) \\ |\pp|= d-2} }
\lv \sum_{i_j} a_\ii x^j_{i_j} \rv_{\pp ' \sep \pp}
\\
=\sum_{j=2}^d \sum_{ \substack{\pp \in \pp([d]\setminus \{j\}) \\ |\pp|=d-2}}
\lv \sum_{i_j}a_\ii x^j_{i_j} \rv_{\emptyset \sep \pp }
+\sum_{j=2}^d \sum_{j\neq k=1}^d \sum_{ \substack{\pp\in \pp([d]\setminus\{j,k\}) \\ |\pp|=d-2}  }
\lv \sum_{i_j} a_\ii x^j_{i_j} \rv_{\{k\}\sep \pp}.
\end{multline*}

\begin{prep}
\label{prop}
For any $d+1\geq 4$, $\eps>0$ and $U\subset (B^n_2)^{d-1}\times T$,
\begin{align}
N\left(U ,\rho_A, \sum_{k=0}^{d-2} \eps^{d-k} s_k(A) +\eps \sup_{(x^2,\ldots,x^d,t) \in U} \hat{\alpha}_A((x^2,\ldots,x^d)) \right) \leq \exp(C(d) \eps^{-2}). \nonumber
\end{align}

\end{prep}

\begin{proof} We will estimate the quantities $W^U_d(\alpha_A,\eps)$ and $V^U_d(\be_A,\eps)$ appearing in Corollary \ref{gorsz}.

Since $U\subset (B^n_2)^{d-1}\times T$,  Jensen's inequality yields for $I \subset \{2,\ldots,d\}$,
\begin{align}\label{prop1}
W^U_I(\alpha_A)&=\sup_{(x^2,\ldots,x^{d},t) \in U} \E \al{\left(\los{2}{d}\right) \ten t} \\
&\leq \sup_{(x^2,\ldots,x^{d},t) \in U}  \sqrt{\E \sum_{i_1}\left( \sum_{i_2,\ldots,i_{d+1}} a_\ii \prod_{k\in I} g^k_{i_k} \prod_{k \in [d]\setminus (I\cup \{1\})} x^k_{i_k}  t_{i_{d+1}} \right)^2 } \nonumber \\
&=\sup_{(x^2,\ldots,x^{d},t) \in U}  \sqrt{\sum_{i_{I\cup\{1\}}}\left( \sum_{i_{[d+1]\setminus (I \cup \{1\})  }} a_\ii \prod_{k \in [d]\setminus (I\cup \{1\})} x^k_{i_k} t_{i_{d+1}}  \right)^2 }\nonumber\\
& \leq \lv A \rv_{\emptyset \sep I\cup\{1\},\{k\}: k \in [d]\setminus (I \cup \{1\}) }\leq s_{d-|I|}(A). \nonumber
\end{align}
By estimating a little more accurately in the second inequality in \eqref{prop1}
we obtain for $2\leq j \leq d$,

\begin{align}
\label{pop1}
W^U_{\{j\}}(\alpha_A)
&\leq \sup_{(x^2,\ldots,x^{d},t) \in U} \sum_{\substack{2\leq l \leq d\\ l\neq j}}\sup_{(y^2,\ldots,y^d)\in (B^n_2)^{d-1}} \sqrt{\sum_{i_1,i_j}\left( \sum_{i_{[d+1]\setminus \{1,j\}  }} a_\ii x^l_{i_l} \prod_{\substack{2\leq k \leq d\\ k\neq j,l}} y^k_{i_k} t_{i_{d+1}}  \right)^2 }
\\
&\leq \sup_{(x^2,\ldots,x^{d},t) \in U}  \sum_{l=2}^d\sum_{\substack{\pp \in \pp([d]\setminus \{l\}) \\ |\pp|=d-2} } \lv \sum_{i_l} a_\ii x^l_{i_l} \rv_{\emptyset\sep\pp}. \nonumber
\end{align}
Observe that \eqref{pop1} is not true for $d+1=3$ (cf. Remark \ref{niezachodzi}).

Let us now pass to the quantity $V^U_d(\be_A,\eps)$. The definition of $V^U_I$ and the inclusion $U\subset (B^n_2)^{d-1}\times T$ yield
\begin{align}
 V^U_I(\beta_A)&\leq \lv A \rv_{\{1\} \{i\} : \ i \in I  \sep \{k\}: k \in [d]\setminus (I \cup \{1\})  } \leq s_{d-|I|-1}(A)\quad \textrm{for } I\neq \emptyset\label{prop2}
\end{align}
and
\begin{align}\label{ost}
V^U_\emptyset(\beta_A) &  \leq \sup_{(x^2,\ldots,x^{d},t) \in U}\sum_{l=2}^d \sup_{(y^2,\ldots,y^d)\in (B^n_2)^{d-1}} \E \sup_{t' \in T} \sum_{\ii} a_\ii g_{i_1} x^l_{i_l} \prod_{\substack{2\leq k \leq d\\ k\neq l}} y^k_{i_k} t'_{i_{d+1}} \\
&\leq \sup_{(x^2,\ldots,x^{d},t) \in U} \sum_{l=2}^d \lv \sum_{i_j} a_\ii x^l_{i_l} \rv_{\{1\} \sep \{k\}:\ k\in [d]\setminus \{1,l\} }. \nonumber
\end{align}
Inequalities \eqref{prop1}-\eqref{ost} imply that
\begin{align*}
&W^U_d(\alpha_A,\eps)+V^U_d(\be_A,\eps)
\\
&=\sum_{k=2}^{d-1} \eps^k \sum_{I\subset \{2,\ldots,d \}: |I|=k} W^U_I(\alpha_A)+\sum_{k=1}^{d-1} \eps^{k+1} \sum_{I\subset \{2,\ldots,d \}: |I|=k} V^U_I(\be_A) \\
&+\eps\left( \sum_{j=2 }^d W^{U}_{\{j\}}(\alpha_A)+ V^{U}_{\emptyset}(\be_A) \right)
\\
&\leq C(d) \sum_{k=0}^{d-2} \eps^{d-k} s_k(A) + 2\eps \sup_{(x^2,\ldots,x^d,t) \in U} \left(\sum_{l=2}^d \sum_{\substack{(\pp,\pp ') \in \pp([d]\setminus \{l\}) \\ |\pp|= d-2} }
\lv \sum_{i_l} a_\ii x^l_{i_l} \rv_{\pp '\sep \pp} \right).
\end{align*}

Hence the assertion is a simple consequence of Corollary \ref{gorsz}.
\end{proof}

\begin{rem}\label{niezachodzi}
Proposition \ref{prop} is not true for $d+1=3$. The problem arises in \eqref{pop1} -- for $d=2$ there does not exist $\mathcal{P} \in \mathcal{P}([d]\setminus \{l\})$  such that $|\mathcal{P}|=d-2$. This is the main reason why proofs for chaoses of order $d=2$ (cf. \cite{wek}) have a different nature than for higher order chaoses.
\end{rem}

\section{Proof of Theorem \ref{maintheorem} }\label{sek4}

We will prove Theorem \ref{maintheorem} by induction on $d$ (recall that matrix $A$ has order $d+1$). To this end we need to amplify the induction thesis.
For $U\subset (\R^{n})^{d-1} \times \R^m$ we define
$$
F_A(U)=\E \sup_{(x^2,\ldots,x^{d+1}) \in U} \sum_\iid a_\iid g_{i_1} \prod_{k=2}^{d+1} x^k_{i_k}.
$$

\begin{twr}\label{oszglow}
For any $U\subset (B^n_2)^{d-1}\times T$ and any $p\geq 1$
\begin{align}
F_A(U) \leq C(d) \left(\sqrt{p} \Delta_A(U)+ \sum_{k=0}^{d-1} p^{\frac{k+1-d}{2}}s_k(A) \right), \label{costam}
\end{align}
where
$$
\Delta_A(U)
= \sup_{(x^2,\ldots,x^d,t),(y^2,\ldots,y^d,t')\in U}\rho_A((x^2,\ldots,x^d,t),(y^2,\ldots,y^d,t'))
=\mathrm{diam}(A,\rho_{A}).
$$

\end{twr}
Clearly it is enough to prove Theorem \ref{oszglow} for finite sets $U$.
Observe that
$$
\Delta_A( (B^n_2)^{d-1}\times T )\leq 2\lv A \rv_{\emptyset \sep \{j\}:j \in [d] }  =2 s_d(A),
$$
thus Theorem \ref{oszglow} implies Theorem \ref{maintheorem}.
We will prove \eqref{costam} by induction on $d+1$, but first we will show several consequences of the theorem. In the next three lemmas, we shall assume that Theorem \ref{oszglow} (and thus also Theorem \ref{thm:proc} ) holds for all matrices of order smaller than $d+1$.

\begin{lem}\label{poprawka}
Let $p\geq 1$, $l\geq 0$ and $d+1\geq 4$. Then
$$N\left((B^n_2)^{d-1},\rhp,2^{-l}\sum_{k=0}^{d-1}p^{\frac{k+1-d}{2}}s_k(A)  \right) \leq \exp(C(d)2^{2l}p ),$$
where $\rhp$ is the distance on $(\R^n)^{d-1}$ corresponding to the norm $\haa$.

\end{lem}

\begin{proof}

Note that
$$
\E \ha{G^2,\ldots,G^d}=\sum_{j=2}^d \sum_{\substack{(\pp,\pp ') \in \pp([d]\setminus \{j\})\\|\pp |=d- 2} } \E \lv \sum_{i_j} a_\ii g_{i_j} \rv_{\pp '\sep \pp}.
$$
Up to a permutation of the indexes we have two possibilities
\begin{equation}\label{3.3.1}
 \lv \sum_{i_j} a_\ii g_{i_j} \rv_{\pp '\sep \pp}=\begin{cases}\lv \sum_{i_j}a_\ii g_{i_j} \rv_{\emptyset\sep\{1,2\},\{ \{l\}:\ 3\leq l \leq d,\ l\neq j \}  } &\textrm{or} \\ \lv \sum_{i_j} a_\ii g_{i_j} \rv_{\{1\}\sep \{l\}: 2\leq l \leq d,\ l \neq j }. \end{cases}
\end{equation}

First assume that $ \lv \sum_{i_j} a_\ii g_{i_j} \rv_{\pp '\sep \pp}=\lv \sum_{i_j}a_\ii g_{i_j} \rv_{\emptyset\sep\{1,2\},\{ \{l\}:\ 3\leq l \leq d,\ l\neq j \}  }$. In this case
$$
\lv \sum_{i_j}a_\ii g_{i_j} \rv_{\emptyset\sep \{1,2\},\{ \{l\}:\ 3\leq l \leq d,\ l\neq j \}  }=\lv \sum_{i_1} b_{i_1,\ldots,i_{d}}g_{i_1} \rv_{\emptyset\sep \{2\},\ldots,\{d-1\} }
$$
for an appropriately chosen matrix $B=(b_{i_1,\ldots,i_{d}})$ (we treat a pair of indices $\{1,2\}$ as a single index and renumerate the indices in such a way that $j$,$\{1,2\}$ and $d+1$ would become $1$,$2$ and $d$ respectively).

Clearly,
\begin{align}
\sum_{\substack{(\pp ', \pp) \in \pp([d-1])\\ |\pp|=k}} \lv B \rv_{\pp '\sep \pp}=\sum_{\substack{(\pp ',\pp) \in \mathcal{C}\\ |\pp|=k}} \lv A \rv_{\pp '
\sep \pp}  \leq \sum_{\substack{(\pp ',\pp) \in \pp([d])\\ |\pp|=k}} \lv A \rv_{\pp '\sep \pp}=s_k(A), \label{locb}
\end{align}
where $\mathcal{C}\subset \pp([d])$ is the set of partitions which do not separate $1$ and $2$.

Thus, Theorem \ref{maintheorem} applied to the matrix $B$ of order $d$ yields
\begin{align}\label{3.3.2}
\E \lv \sum_{i_j}a_\ii g_{i_j} \rv_{\emptyset\sep \{1,2\},\{ \{l\}:\ 3\leq l \leq d,\ l\neq j \}  }=\E\lv \sum_{i_1} b_{i_1,\ldots,i_{d}}g_{i_1} \rv_{\emptyset\sep \{2\},\ldots,\{d-1\} } \\
\leq C(d)\sum_{(\pp ', \pp) \in \pp([d-1])}p^{\frac{|\pp |+2-d}{2}}\nor{B}\leq C(d) \sum_{k=0}^{d-1} p^{\frac{k+2-d}{2}} s_k(A). \nonumber
\end{align}
Now assume that $\lv \sum_{i_j} a_\ii g_{i_j} \rv_{\pp '\sep \pp}= \lv \sum_{i_j} a_\ii g_{i_j} \rv_{\{1\}\sep \{l\}: 2\leq l \leq d,\ l \neq j }$ and observe that
\begin{align*}
\E \lv \sum_{i_j} a_\ii g_{i_j} \rv_{\{1\}\sep\{l\}: 2\leq l \leq d,\ l \neq j }&=\E^g \sup_{x^l \in B^n_2,\ 2\leq l \leq d,\ l \neq j } \E^{g'} \sup_{t \in T} \sum_\ii a_\ii g'_{i_1}g_{i_j} \prod_{2\leq l \leq d,\ l \neq j} x^l_{i_l} t_{i_{d+1}}\\
&=\E \sup_{x^l \in B^n_2,\ 2\leq l \leq d,\ l \neq j } \sup_{m\in \mathcal{M}} \sum_\ii a_\ii g_{i_j}  \prod_{2\leq l \leq d,\ l \neq j} x^l_{i_l} m_{i_1,i_{d+1}} \\
&=\E \sup_{x^l \in B^n_2,\ 2\leq l \leq d-1 } \sup_{m \in \widetilde{\mathcal{M}}} \sum_{i_1,\ldots,i_d} d_{i_1,\ldots,i_d} g_{i_1} \prod_{l=2}^{d-1} x^l_{i_l} m_{i_d},
\end{align*}
where $D=(d_{i_1,\ldots,i_d})_{i_1,\ldots,i_d}$ is an appropriately chosen matrix of order $d$, the set $\mathcal{M} \subset \R^n\otimes \R^m $ satisfies
\begin{displaymath}
 \E \sup_{t \in T} \sum_{i,j} b_{i,j} g_i t_j=\sup_{m \in \mathcal{M}} \sum_{i,j} b_{i,j} m_{i,j}\textrm{ for any matrix }  (b_{i,j})_{i\le n, j\leq m},
 \end{displaymath}
and $\widetilde{\mathcal{M}}$ corresponds to $\mathcal{M}$ under a natural identification of $\R^n \otimes \R^m$ with $\R^{nm}$.

Applying Theorem \ref{maintheorem} to the matrix $D$ of order $d$ gives
\begin{align}\label{3.3.3}
\E \lv \sum_{i_j} a_\ii g_{i_j} \rv_{\{1\}\sep \{l\}: 2\leq l \leq d,\ l \neq j }
&=\E \sup_{x^l \in B^n_2,\ 2\leq l \leq d-1 } \sup_{m \in \widetilde{\mathcal{M}}}\sum_{i_1,\ldots,i_d} d_{i_1,\ldots,i_d} g_{i_1} \prod_{l=2}^{d-1} x^l_{i_l} m_{i_d}
\\
&\leq C(d) \sum_{(\pp ',\pp)\in \pp([d-1])} p^{\frac{|\mathcal{P}|+2-d}{2}}\nor{D}^{\widetilde{\mathcal{M}}} \nonumber
\\
&\leq C(d) \sum_{(\pp ',\pp)\in \pp([d])}  p^{\frac{|\mathcal{P}|+2-d}{2}}\nor{A}
\nonumber
\\
&=C(d) \sum_{k=0}^{d-1} p^{\frac{k+2-d}{2}}s_k(A),
\nonumber
\end{align}
where $\nor{D}^{\widetilde{\mathcal{M}}}$ is defined in the same manner as $\nor{A}$
(see \eqref{eq:new-norm-definition}) but the supremum is taken over the set  $\widetilde{\mathcal{M}}$
instead of $T$.
The second inequality in \eqref{3.3.3} can be justified analogously as \eqref{locb}.

Combining \eqref{3.3.1},\eqref{3.3.2}, \eqref{3.3.3} and  the dual Sudakov inequality (Theorem \ref{thm:dualSud},
note that $(B_2^n)^{d-1}\subseteq \sqrt{d-1}B_2^{n(d-1)}$)  we obtain
\begin{multline*}
N\left((B^n_2)^{d-1},\hat{\rho}_A,t \sum_{k=0}^{d-1} p^{\frac{k+2-d}{2}} s_k(A)\right) \\
\leq N\left((B^n_2)^{d-1},\hat{\rho}_A,C(d)^{-1}t\E \haa(G^2,\ldots,G^n) \right) \leq \exp(C(d)t^{-2}).
\end{multline*}
It is now enough to choose $t=(\sqrt{p} 2^l)^{-1}$.
\end{proof}

From now on for $U \subseteq (\R^{n})^{d-1} \times \R^m$ we denote
\begin{displaymath}
  \hat{\alpha}_A(U) = \sup_{(x^2,\ldots,x^d,t) \in U}\ha{(x^2,\ldots,x^d)}.
\end{displaymath}

\begin{lem}
\label{przesun}
Suppose that $d+1\geq 4$, $\textbf{y}=(y^2,\ldots,y^d)\in (B^n_2)^{d-1}$ and $U \subset (B^n_2)^{d-1} \times T$. Then for any $p\geq 1$ and $l\geq 0$, we can find a decomposition
$$
U=\bigcup_{j=1}^N U_j, \ N\leq \exp(C(d)2^{2l}p)
$$
such that for each $j\leq N$,
\begin{align}
\label{6}
F_A((\textbf{y},0)+U_j)
\leq F_A(U_j)+C(d)\left(\ha{\textbf{y}}+ \hat{\alpha}_A(U) +2^{-l}\sum_{k=0}^{d-2} p^{\frac{k+1-d}{2}}s_k(A) \right)
\end{align}
and
\begin{align}
\label{7}
\Delta_A(U_j)\leq 2^{-l} p^{-1/2} \hat{\alpha}_A(U)+2^{-2l}\sum_{k=0}^{d-2} p^{\frac{k-d}{2}}s_k(A).
\end{align}
\end{lem}

\begin{proof}
Fix $\textbf{y}\in (B^n_2)^{d-1}$ and $U \subset (B^n_2)^{d-1} \times T$.
For $I\subset\{2,\ldots,d\}$, $\xx=(x^2,\ldots,x^d,t),\xxx=(\tilde{x}^2,\ldots,\tilde{x}^d,t') \in  (\R^{n})^{d-1} \times \R^m$ and $S\subset  (\R^{n})^{d-1} \times \R^m$, we define
\[
\rh(\xx,\xxx):=\sqrt{\sum_{i_1} \left(\sum_{i_2,\ldots,i_{d+1}} a_\ii \prod_{k \in I} y^k_{i_k} \left(t_{i_{d+1}}\prod_{ \stackrel{2\leq j\leq d}{j \notin I}} x^j_{i_j}-t'_{i_{d+1}}\prod_{ \stackrel{2\leq j\leq d}{j \notin I}} \tilde{x}^j_{i_j} \right) \right)^2 },
\]
\[
\sr(S):=\sup \left\{\rh(\xx,\xxx):\xx,\xxx \in S \right\}
\]
and
\[
\ff(S):=\E \sup_{(x^2,\ldots,x^d,t) \in S} \sum_\ii a_\ii g_{i_1}\prod_{k \in I} y^k_{i_k} \left(\prod_{ \stackrel{2\leq j\leq d}{j \notin I}} x^j_{i_j} \right) t_{i_{d+1}}.
\]

If $I=\{2,\ldots,d\}$ then for $S\subset (B^n_2)^{d-1}\times T$ we have
\begin{align}\label{3.4.1}
F^{\textbf{y},\{2,\ldots,d\}}_A(S)&\leq \E \sup_{t \in T} \sum_\ii a_\ii g_{i_1}\prod_{k =2}^d y^k_{i_k} t_{i_{d+1}}\\
&\leq \sup_{(x^2,\ldots,x^{d-1}) \in (B_2^n)^{d-2}} \E \sup_{t \in T} \sum_\ii a_\ii g_{i_1}\left(\prod_{ j=2}^{d-1} x^j_{i_j} \right) y^d_{i_d} t_{i_{d+1}} \nonumber\\
&= \lv\sum_{i_d} a_\ii y^d_{i_d} \rv_{\{1\}\sep\{k\} \ : \ k=2,\ldots,d-1 }   \leq \ha{\textbf{y}}. \nonumber
\end{align}

If $I\neq \emptyset,\{2,\ldots,d\}$ then Theorem \ref{oszglow} applied to the matrix
$$
A(\textbf{y},I):=\left( \sum_{i_I} a_\ii \prod_{k \in I} y^k_{i_k} \right)_{\ii_{[d+1]\setminus I}}
$$
of order $d-|I|+1<d+1$ gives for any $S\subset (B^n_2)^{d-1}\times T$ and $q\geq 1$,
\begin{align*}
\ff(S)\leq C(d-|I|) \left(q^{1/2}\sr{(S)}+\sum_{k=0}^{d-|I|-1} q^{\frac{k+1-d+|I|}{2}}s_k(A(\y,I))  \right).
\end{align*}

For any $2\leq k\leq d$,  $y^k \in B^n_2$ thus $s_k(A(\y,I))\leq s_{k+|I|}(A)$ for $k<d-|I|-1$ and $s_{d-|I|-1}(A(\y,I))\leq \ha{\y}$.
Hence,
\begin{align}
\ff(S)\leq C(d-|I|) \left(q^{1/2}\sr(S)+\ha{\textbf{y}}+\sum_{k=0}^{d-2} q^{\frac{k+1-d}{2}}s_k(A)\right). \label{3.4.2}
\end{align}
By the triangle inequality,
\begin{align*}
F_A((y^2,\ldots,y^d,0)+S)-F_A(S)\leq \sum_{\emptyset\neq I\subset \{2,\ldots,d\}} \ff(S).
\end{align*}
Combining \eqref{3.4.1} and \eqref{3.4.2} we obtain for $S\subset (B^n_2)^{d-1}\times T$ and $q\geq 1$,
\begin{multline}
F_A((y^2,\ldots,y^d,0)+S)\\
\leq F_A(S)+C(d)\left(\ha{\textbf{y}}+\sum_{\emptyset\neq I\subsetneq \{2,\ldots,d\}}  q^{1/2}\sr(S)+\sum_{k=0}^{d-2} q^{\frac{k+1-d}{2}}s_k(A) \right). \label{3.4.4}
\end{multline}

Fix $I\subsetneq  \{2,\ldots,d\}$, $|I| < d-2$ (we do not exclude $I=\emptyset$). Taking supremum over $\textbf{y} \in (B^n_2)^{d-1}$ we conclude that

\begin{align*}
\sup_{(x^2,\ldots,x^d,t) \in U}\hat{\alpha}_{A(\y,I)}(\{x^k\}\ : \ k\in \{2,\ldots,d\}\setminus I)\leq \sup_{(x^2,\ldots,x^d,t) \in U}\haa((x^2,\ldots,x^d)).
\end{align*}


Recall also that $s_k(A(\y,I))\leq s_{k+|I|}(A)$, thus we may apply $2^{d-1}-d$ times Proposition \ref{prop} with $\varepsilon=2^{-l} p^{-1/2}$ and find a decomposition $U=\bigcup_{j=1}^{N_1} U_j'$, $N_1\leq \exp(C(d)2^{2l}p)$ such that for each $j$ and $I\subset \{2,\ldots,d\}$ with $|I| < d-2$,
\begin{align}\label{3.4.5}
\sr(U_j')\leq 2^{-l}p^{-1/2} \ha{U}+2^{-2l} \sum_{k=0}^{d-2} p^{\frac{k-d}{2}}s_k(A).
\end{align}

If $|I| = d-2$ then the distance $\rho_A^{\mathbf{y},I}$ corresponds to a norm $\alpha_{A(\mathbf{y},I)}$ on $\R^{nm}$ given by
\begin{displaymath}
  \alpha_{A(\mathbf{y},I)}(\mathbf{x}) = \sqrt{\sum_{i_1} \left(\sum_{i_2,\ldots,i_{d+1}} a_\ii x_{i_{\{j,d+1\}}} \prod_{k \in I}y_{i_k}^k \right)^2 },
\end{displaymath}
where $j$ is defined by the condition $\{1,j\} = [d]\setminus I$ (cf. \eqref{eq:distance-definition}, \eqref{eq:norm-definition}  and observe that
$A(\mathbf{y},I)$ is an $n\times m$ matrix). We define also (as in \eqref{eq:definition-beta})
$$
\beta_{A(\mathbf{y},I)}(\mathbf{x})=\E \sup_{t \in T} \sum_{\ii} a_\ii g_{i_1} x_{i_j} \prod_{l \in I} y^l_{i_l} t_{i_{d+1}}.
$$

Recall the definitions \eqref{loc}  and \eqref{eq:definition-V} and note that (denoting by $\widetilde{U}$ the projection of $U$ onto the $j$-th and $(d+1)$-th coordinate)

\begin{align}\label{eq:klopotliwe-I}
  W^{\widetilde{U}}_{2}( \alpha_{A(\mathbf{y},I)(\mathbf{x})}, \varepsilon) & = \varepsilon \sup_{(x^2,\ldots,x^d,t)\in U} \E \sqrt{\sum_{i_1} \left(\sum_{i_2,\ldots,i_{d+1}} a_\ii g_{i_j}t_{i_{d+1}} \prod_{k \in I} y_{i_k}^k \right)^2 }\\
  & \le \eps \sup_{(x^2,\ldots,x^d,t) \in U} \sqrt{\sum_{i_1,i_j} \left(\sum_{\ii_{I}} a_\ii \prod_{k\in I} y^k_{i_k} t_{i_{d+1}}  \right)^2} \leq  \eps \hat{\alpha}(\mathbf{y}). \nonumber
\end{align}
where we again used that  $y_k \in B_2^n$, $U \subset (B_2^n)^{d-1}\times T$.

We also have
\begin{align*}
  V_{\{j\}}^{\widetilde{U}}(\beta_{A(\mathbf{y},I)}) = \E \sup_{t \in T}\sum_\ii a_\ii g_{i_1}^1 g_{i_j}^2 t_{i_k} \prod_{k\in I}y_{i_k}^k \le s_{d-2}(A)
\end{align*}
and
\begin{align*}
V_\emptyset^{\widetilde{U}}(\beta_{A(\mathbf{y},I)}) = \sup_{(x^2,\ldots,x^d,t)\in U} \E \sup_{t' \in T} \sum_\ii a_\ii g_{i_1} x_{i_j}^j t_{i_{d+1}}'\prod_{k\in I}y_{i_k}^k \le  \hat{\alpha}(\mathbf{y}).
\end{align*}

Thus
\begin{displaymath}
V^{\widetilde{U}}_2(\beta_{A(\mathbf{y},I)},\varepsilon) \le \varepsilon \hat{\alpha}(\mathbf{y}) + \varepsilon^2 s_{d-2}(A).
\end{displaymath}

Taking $\varepsilon = 2^{-l-1}p^{-1/2}$ and combining the above estimate with \eqref{eq:klopotliwe-I} and Corollary \ref{gorsz} (applied $d-1$ times) we obtain a partition $U = \bigcup_{j=1}^{N_2} U_j''$ with $N_2 \le \exp(C(d)2^{2l}p)$ and
\begin{align}\label{eq:klopotliwe-zbiory-dodatkowe-rozbicie}
  \sr(U_j'')\le 2^{-l}p^{-1/2} \hat{\alpha}_A(\mathbf{y}) + 2^{-2l}p^{-1} s_{d-2}(A)
\end{align}
for any $I \subset \{2,\ldots,d\}$ with $|I| = d-2$ and $j \le N_2$.

Intersecting the partition $(U_i')_{i \leq N_1}$ (which fullfills \eqref{3.4.5}) with $(U''_j)_{j\leq N_2}$ we obtain a partition $U = \bigcup_{i=1}^N U_i$ with $N \le N_1N_2 \le \exp(C(d)2^{2l}p)$ and such that for every $i\le N$ there exist $j \le N_1$ and $l \le N_2$ such that $U_i \subset U_j'\cap U_l''$.

Inequality \eqref{6} follows by \eqref{3.4.4} with $q=2^{2l}p$, \eqref{3.4.5} and \eqref{eq:klopotliwe-zbiory-dodatkowe-rozbicie}. Observe that \eqref{7} follows by \eqref{3.4.5} for $I=\emptyset$.
\end{proof}

\begin{lem}\label{final}
Suppose that $U$ is a finite subset of $(B^n_2)^{d-1}\times T$, with $|U|\geq 2$ and $U-U\subset (B^n_2)^{d-1} \times (T-T)$. Then for any $p\geq 1, l \ge 0$ there exist finite sets $U_i\subset  (B^n_2)^{d-1} \times T$ and $(\y_i,t_i) \in U$, $i=1,\ldots,N$ such that
\begin{enumerate}[(i)]
\item $2\leq N \leq \exp(C(d)2^{2l}p)$,
\item $U=\bigcup_{i=1}^N ((\y_i,0)+U_i), (U_i-U_i) \subset U-U,|U_i|\leq |U|-1$,
\item $\Delta_A(U_i)\leq 2^{-2l}\sum_{k=0}^{d-1}p^{\frac{k-d}{2}}s_k(A)  $,
\item $\ha{U_i}\leq 2^{-l} \sum_{k=0}^{d-1}p^{\frac{k+1-d}{2}}s_k(A)$,
\item $F_A((\y_i,0)+U_i)\leq F_A(U_i)+C(d)\left(\ha{U}+ 2^{-l} \sum_{k=0}^{d-1}p^{\frac{k+1-d}{2}}s_k(A)\right)$.
\end{enumerate}

\end{lem}
\begin{proof}

By Lemma \ref{poprawka}  we get
$$
(B^n_2)^{d-1}=\bigcup_{i=1}^{N_1} B_i,\ N_1\leq \exp(C(d)2^{2l}p),
$$
where the diameter of the sets $B_i$ in the norm $\hat{\alpha}$ satisfies
$$
{\rm diam}(B_i,\hat{\alpha}_A) \le 2^{-l}\sum_{k=0}^{d-1}p^{\frac{k+1-d}{2}}s_k(A).
$$
Let $U_i=U\cap (B_i\times T)$. Selecting arbitrary $(\mathbf{y}_i,t_i) \in U_i$ (we can assume that these sets are nonempty) and using Lemma \ref{przesun}  (with $l+1$ instead of $l$) we decompose $U_i-(\y_i,0)$ into $\bigcup_{j=1}^{N_2} U_{ij} $ in such a way that $N_2\leq \exp(C(d)2^{2l}p)$,
\begin{align*}
F_A((\y_i,0)+U_{ij})
&\leq F_A(U_{ij})+C(d)\Bigg(\ha{\textbf{y}_i}+\ha{U_i-(\y_i,0)}+2^{-l}\sum_{k=0}^{d-2}p^{\frac{k+1-d}{2}}s_k(A) \Bigg)
\\
&\leq F_A(U_{ij})+C(d)\Bigg(\ha{\textbf{y}_i}+{\rm diam}(B_i,\hat{\alpha}_A) +2^{-l}\sum_{k=0}^{d-1}p^{\frac{k+1-d}{2}}s_k(A)\Bigg)
\\
&\leq  F_A(U_{ij})+C(d)\Bigg(\ha{U}+2^{-l}\sum_{k=0}^{d-1}p^{\frac{k+1-d}{2}}s_k(A) \Bigg)
\end{align*}
and
\begin{align*}
\Delta_A(U_{ij})\leq 2^{-l-1} p^{-1/2} \ha{U_i-(\y_i,0)}+2^{-2l-2}\sum_{k=0}^{d-2}p^{\frac{k-d}{2}}s_k(A)
\leq 2^{-2l}\sum_{k=0}^{d-1}p^{\frac{k-d}{2}}s_k(A)  .
\end{align*}
We take the decomposition $U=\bigcup_{i,j} ((\y_i,0)+U_{ij}).$
We have $N=N_1N_2\leq \exp(C(d)2^{2l}p)$. Without loss  of generality we can assume $N\geq 2$ and $|U_{i,j}|\leq |U|-1$. Obviously, $U_{ij}-U_{ij}\subset U_i-U_i \subset U-U$ and $\ha{U_{ij}}\leq \ha{U_i-(\y_i,0)}\leq 2^{-l} \sum_{k=0}^{d-1}p^{\frac{k+1-d}{2}}s_k(A) $. A relabeling of the obtained decomposition concludes the proof.
\end{proof}
\begin{proof}[Proof of Theorem \ref{oszglow}]

In the case of $d+1=3$ Theorem \ref{oszglow} is proved in \cite{wek} (see Remark 37 therein).

Assuming \eqref{costam}  to hold  for all matrices of order $\{3,4,\ldots,d\}$, we will prove it for matrices of order $d+1\geq 4$. Let $U\subset (\R^n)^{d-1} \times \R^m$ and let us put $\Delta_0=\Delta_A(U)$, $\dd_0=\ha{(B^n_2)^{d-1}\times T}\leq C(d) s_{d-1}(A)$,
\begin{align*}
\Delta_l&:=2^{2-2l} \sum_{k=0}^{d-1}p^{\frac{k-d}{2}}s_k(A), \ \dd_l:=2^{1-l} \sum_{k=0}^{d-1}p^{\frac{k+1-d}{2}}s_k(A) \textrm{ for } l\geq 1.
\end{align*}

Suppose first that $U\subset (\frac{1}{2} (B^n_2)^{d-1})\times T$ and define
\begin{align*}
c_U(r,l):=\sup \Big\{F_A(S)\colon\ &S\subset (B^n_2)^{d-1}\times T, S-S \subset U-U,
\\
&|S|\leq r,\Delta_A(S)\leq \Delta_l,\ha{S}\leq \dd_l \Big\}.
\end{align*}
Note that any subset $S\subset U$ satisfies $\Delta_A(S)\leq \Delta_0$ and $\ha{S}\leq \dd_0$, therefore,
\begin{align} \label{k.0}
c_U(r,0)\geq \sup \{F_A(S):S\subset U,|S|\leq r \}.
\end{align}
 We will now show that for $r\geq 2$,
\begin{align} \label{k.1}
c_U(r,l)\leq c_U(r-1,l+1)+C(d)\left(\dd_l+2^l\sqrt{p}\Delta_l+2^{-l}\sum_{k=0}^{d-1}p^{\frac{k+1-d}{2}}s_k(A)   \right).
\end{align}
Indeed, let us take $S\subset (B^n_2)^{d-1}\times T$ as in the definition of $c_U(r,l)$. Then by Lemma \ref{final} we may find a decomposition $S=\bigcup_{i=1}^N ((\y_i,0)+S_i)$ satisfying $(i)-(v)$
with $U$, $U_i$ replaced by $S$, $S_i$. Hence, by Lemma \ref{rozbij},  we have
\begin{align}
\label{k.2}
F_A(S)
&\leq C\sqrt{\log N} \Delta_A(S)+\max_i F_A((\y_i,0)+S_i)
\\
\nonumber
&\leq C(d) \left(2^l\sqrt{p}\Delta_l+\ha{S}+2^{-l}\sum_{k=0}^{d-1}p^{\frac{k+1-d}{2}}s_k(A) \right)+\max_iF_A(S_i).
\end{align}

We have $\Delta_A(S_i)\leq \Delta_{l+1}$, $\ha{S_i} \leq \hat{\Delta}_{l+1},\  S_i-S_i\subset S-S \subset U-U$ and $|S_i|\leq |S|-1\leq r-1$, thus $\max_i F_A(S_i)\leq c_U(r-1,l+1)$ and \eqref{k.2} yields \eqref{k.1}.
Since $c_U(1,l)=0$, \eqref{k.1} yields
\begin{align*}
c_U(r,0)\leq C(d)\sum_{l=0}^\infty \left(\dd_l+ 2^l \sqrt{p} \Delta_l+2^{-l}\sum_{k=0}^{d-1}p^{\frac{k+1-d}{2}}s_k(A)  \right).
\end{align*}
For $U\subset (\frac{1}{2} (B^n_2)^{d-1})\times T$, we have by \eqref{k.0}
\begin{align*}
F_A(U)
&=\sup \{F_A(S): S\subset U, |S|< \infty \} \leq \sup_r c_U(r,0)
\\
&\leq C(d) \left(\sqrt{p} \Delta_A(U)+\sum_{k=0}^{d-1}p^{\frac{k+1-d}{2}}s_k(A) \right).
\end{align*}
Finally, if $U\subset  (B^n_2)^{d-1}\times T$, then $U' := \{(\mathbf{y}/2,t)\colon (\mathbf{y},t) \in U\} \subset (\frac{1}{2}  (B^n_2)^{d-1})\times T$ and $\Delta_A(U')=2^{1-d} \Delta_A(U)$, hence,
\begin{align*}
F_A(U)=2^{d-1}F_A(U')\leq C(d)\left(\sqrt{p} \Delta_A(U)+\sum_{k=0}^{d-1}p^{\frac{k+1-d}{2}}s_k(A)  \right).
\end{align*}

\end{proof}

\section{Proofs of main results}\label{sek5}
\label{sec:proofs}
We return to the notation used Section \ref{dwa2}. In particular in this section the multi-index $\ii$ takes values in $[n]^d$
(instead of $[n]^{d}\times[m]$ as we had in the two previous sections) and all summations over $\ii$ should be understood as summations over $[n]^d$.

\subsection{Proofs of Theorems \ref{thm:gl} and \ref{thm:og}}

\begin{proof}[Proof of Theorem \ref{thm:gl}]
We start with the lower bound. Fix $J\subset [d]$, $\pp \in \pp([d]\setminus J)$ and observe that
\begin{align*}
 \lv  \sum_\ii a_\ii \prod_{k=1}^d g^k_{i_k}  \rv_p & \geq \left(\E^{(G^j):j\in J} \sup_{\substack{\varphi \in F^*\\ \lv \varphi \rv\leq 1}} \E^{(G^j):j\in [d]\setminus J} \left|\varphi\left(\sum_\ii a_\ii \prod_{k=1}^d g^k_{i_k}  \right)\right|^p  \right)^{1/p}
\\
&\geq c(d) \left( \E^{(g^j):j\in J}  p^{\frac{p|\pp|}{2}}
\left| \! \left| \! \left|\left( \sum_{i_J} a_\ii \prod_{j \in J} g^j_{i_j}\right)_{i_{[d]\setminus J}} \right| \! \right| \! \right|_{\mathcal{P}}^p  \right)^{1/p}\\
&\geq c(d) p^{\frac{|\pp|}{2}} \E \left| \! \left| \! \left|\left( \sum_{i_J} a_\ii \prod_{j \in J} g^j_{i_j}\right)_{i_{[d]\setminus J}} \right| \! \right| \! \right|_{\mathcal{P}}  = c(d) p^{\frac{|\pp|}{2}} \altnor{ A },
\end{align*}
where $F^*$ is the dual space and in the second inequality we used Theorem \ref{momrealgaus}.

The upper bound will be proved by an induction on $d$. For $d+1=3$ it is showed in \cite{wek}. Suppose that $d+1\geq 4$ and the estimate holds for matrices of order $\{2,3,\ldots,d\}$. By the induction assumption, we have
\begin{align}
\lv  \sum_\ii a_\ii \prod_{k=1}^d g^k_{i_k}  \rv_p \leq  C(d)\sum_{(\pp,\pp ') \in \pp([d-1])} p^{\frac{|\pp|}{2}} \lv \lv\sum_{i_{d}} a_\ii g_{i_{d}}\rv_{\pp '\sep\pp} \rv_p. \label{4.1.0}
\end{align}
Since $\nor{\cdot}$ is a norm Lemma \ref{kon} yields
\begin{align}
\lv \lv\sum_{i_{d}} a_\ii g_{i_{d}}\rv_{\pp '\sep\pp} \rv_p \leq C \E \lv\sum_{i_{d}} a_\ii g_{i_{d}}\rv_{\pp '\sep\pp} +C \sqrt{p} \lv A \rv_{\pp '\sep \pp \cup \{d\} }. \label{aal}
\end{align}

Choose $\pp=(I_1\ldots,I_k) ,\pp '=(J_1,\ldots,J_m)$ and denote  $J=\bigcup \pp '$. By the definition of $\nor{A}$ we have
\begin{align}
&\lv\sum_{i_{d}} a_\ii g_{i_{d}}\rv_{\pp '\sep\pp}
\label{tosup}
\\
\nonumber
&=\sup \left\{ \E^{(G^1,\ldots,G^m)} \lv \sum_\ii \ai \x{k} \prod_{l=1}^m g^{l}_{i_{J_l}} g^d_{i_d} \rv \ \Big{|}\ \forall_{j=1\ldots,k} \sum_{i_{I_j}} \left(x^{(j)}_{i_{I_j}}\right)^2=1 \right\}
\\
\nonumber
&=\sup \left\{ \left|\!\left|\!\left|\left(\sum_{i_{[d]\setminus J}}\ai \x{k} g^d_{i_d}\right)_{i_J}
\right|\!\right|\!\right|  \ \Big{|}\ \forall_{j=1\ldots,k} \sum_{i_{I_j}} \left(x^{(j)}_{i_{I_j}}\right)^2=1 \right\},
\end{align}
where  $G^l=(g_{i_{J_l}})_{i_{J_l}}$ and $\robn{\ \cdot \ }$ is a norm on $F^{ n^{|J|}}$ given by
$$\robn{(a_{i_{J}})_{i_J}}=\E \lv \sum_{i_J} a_{i_J} \prod_{l=1}^m g^{l}_{i_{J_l}} \rv.$$

Theorem \ref{thm:proc} implies that
\begin{align}
\nonumber
\E &\sup \left\{  \left|\!\left|\!\left|\left(\sum_{i_{[d]\setminus J}}\ai \x{k} g^d_{i_d}\right)_{i_J}\right|\!\right|\!\right|   \  \Big{|}\ \forall_{j=1\ldots,k} \sum_{i_{I_j}} \left(x^{(j)}_{i_{I_j}}\right)^2=1 \right\}
\\
\nonumber
&\leq C(k)\sum_{(\rr ', \rr ) \in \pp([d]\setminus J)} p^{\frac{|\rr |-k}{2}} \robn{A}_{\rr '\sep\rr}
= C(k) \sum_{(\rr ', \rr ) \in \pp([d]\setminus J)}  p^{\frac{|\rr |-k}{2}} \lv A \rv_{\rr '\cup \pp '\sep\rr}
\\
\nonumber
&\leq C(k) \sum_{(\rr ', \rr) \in \pp([d])} p^{\frac{|\rr |-k}{2}} \lv A \rv_{\rr '\sep\rr},
\end{align}
where $\robn{A}_{\rr '\sep\rr}$ is defined as $\lv A \rv_{\rr '\sep\rr}$ but under the expectation occurs the norm $\robn{\cdot}$.

The above and \eqref{tosup} yield
\begin{equation}
\E \lv\sum_{i_{d}} a_\ii g_{i_{d}}\rv_{\pp '\sep\pp}
\leq C(k) \sum_{(\rr ', \rr) \in \pp([d])} p^{\frac{|\rr |-k}{2}} \lv A \rv_{\rr '\sep\rr}. \label{kon2}
\end{equation}
Since $| \pp |=k$ the theorem follows from \eqref{4.1.0},\eqref{aal} and \eqref{kon2}.
\end{proof}

\begin{proof}[Proof of Theorem \ref{thm:og}]

Let $S=\lv \sum a_\ii \pg \rv$. Chebyshev's inequality and Theorem \ref{thm:gl} yield for $p > 0$,
\begin{align}
\Pro \left(S\geq C(d) \sum_{(\pp,\pp ') \in \pp([d])}  p^{|\pp | /2} \nor{A} \right)&\leq e^{1-p}.
\end{align}
Now we substitute
\[
t= C(d)\sum_{\pp ' \in \pp([d])}\|A\|_{\pp' \sep \emptyset}
+C(d)\sum_{\substack{(\pp,\pp ') \in \pp([d])\\ |\pp|\geq 1}}  p^{|\pp | /2} \nor{A}:=t_1+t_2
\]
and observe that if $t_1<t_2$ then
$$
p\geq
\frac{1}{C(d)} \min_{\substack{(\pp,\pp ') \in \pp([d])\\ |\pp|>0}} \left(\frac{t}{\nor{A}} \right)^{2/|\pp|} .
$$
The first inequality of the theorem follows then by adjusting the constants.

On the other hand by the Paley-Zygmund inequality we get for $p \ge 2$,
\begin{align*}
\Pro\left(S\geq C^{-1}(d)\sum_{J\in [d]} \sum_{\pp\in \pp(J)} p^{|\pp|/2}\altnor{A}\right)
&\geq \Pro\left(S^p\geq \frac{1}{2^p} \E S^p \right)
\\
&\geq \left(1-\frac{1}{2^p}\right)^2 \frac{(\E S^p)^2}{\E S^{2p}}\geq e^{-C(d)p},
\end{align*}
where in the last inequality we used Theorem \ref{gaushiper}. The inequality follows by  a similar substitution as
for the upper bound.
\end{proof}

\subsection{Proof of Proposition \ref{prop:splitting-derivatives} and Theorem \ref{thm:general-polynomials}}

Let us first note that Proposition \ref{prop:splitting-derivatives} reduces \eqref{eq:general-poly-1} of Theorem \ref{thm:general-polynomials} to the lower estimate given in Theorem \ref{thm:gl}, while \eqref{eq:general-poly-3} is reduced to Corollary \ref{cor:special-spaces}. The tail bounds \eqref{eq:general-poly-2} and \eqref{eq:general-poly-4} can be then obtained by Chebyshev's and Paley-Zygmund inequalities as in the proof of Theorem \ref{thm:og}. The rest of this section will be therefore devoted to the proof of Proposition \ref{prop:splitting-derivatives}.

The overall strategy of the proof is similar to the one used in \cite{AW} to obtain the real valued case of Theorem \ref{thm:general-polynomials}. It relies on a reduction of inequalities for general polynomials of degree $D$ to estimates for decoupled chaoses of degree $d=1,\ldots,D$. To this end we will approximate general polynomials by tetrahedral ones and split the latter into homogeneous parts of different degrees, which can be decoupled. The splitting may at first appear crude but it turns out that up to constants depending on $D$ one can in fact invert the triangle inequality, which is formalized in the following result due to Kwapie\'n (see \cite[Lemma 2]{kwa}). Recall that a multivariate polynomial is called tetrahedral, if it is affine  in each variable.

\begin{twr}\label{thm_Kwapien}
If $X = (X_1,\ldots,X_n)$ where $X_i$ are independent symmetric random variables, $Q$ is a multivariate tetrahedral polynomial of degree $D$ with coefficients in a Banach space $E$ and $Q_d$ is its homogeneous part of degree $d$, then for any symmetric convex function $\Phi \colon E \to \R_+$ and any $d \in \{0,1, \ldots, D\}$,
\begin{displaymath}
\E\Phi(Q_d(X)) \le \E\Phi(C_D Q(X)). 
\end{displaymath}
\end{twr}

It will be convenient to have the polynomial $f$ represented as a combination of multivariate Hermite polynomials:
\begin{equation}\label{eq:f-as-Hermite}
  f(x_1, \ldots, x_n) = \sum_{d=0}^D \sum_{\rd \in \Delta_d^n} a_\rd h_{d_1}(x_1) \cdots h_{d_n}(x_n),
\end{equation}
where
\[
  \Delta_d^n = \{ \rd = (d_1, \ldots, d_n) \colon \forall_{k \in [n]}\ d_k \ge 0 \text{ and } d_1 + \cdots + d_n = d \}
\]
and $h_m(x) = (-1)^m e^{x^2/2} \frac{d^m}{dx^m} e^{-x^2/2}$ is the $m$-th Hermite polynomial. Recall that Hermite polynomials are orthogonal with respect to the standard Gaussian measure, in particular if $g$ i a standard Gaussian variable, then for $m \ge 1$, $\E h_m(g) = 0$ (we will use this property several times without explicitly referring to it).

In what follows, we will use the following notation. For a set $I$, by $I\uu{k}$ we will denote the set of all one-to-one sequences of length $k$ with values in $I$.
For an $F$-valued $d$-indexed matrix $A = (a_{i_1,\ldots,i_d})_{i_1,\ldots,i_d\le n}$ and $x \in \R^{n^d} \simeq (\R^n)^{\otimes d}$ we will denote
\begin{displaymath}
  \langle A,x\rangle = \sum_{i_1,\ldots,i_d} a_{i_1,\ldots,i_d} x_{i_1,\ldots,i_d}.
\end{displaymath}

Let $(W_t)_{t \in [0,1]}$ be a standard Brownian motion. Consider standard Gaussian random variables $g = W_1$ and, for any positive integer $N$,
\[
  g_{j,N} = \sqrt{N} (W_{\frac{j}{N}} - W_{\frac{j-1}{N}}), \quad j = 1, \ldots, N.
\]
For any $d \ge 0$, we have the following representation of $h_d(g) = h_d(W_1)$ as a multiple stochastic integral
(see~\cite[Example 7.12 and Theorem 3.21]{JansonGHS}),
\[
  h_d(g) = d! \int_0^1 \! \int_0^{t_d} \! \cdots \! \int_0^{t_2} \, dW_{t_1} \cdots dW_{t_{d-1}} dW_{t_d}.
\]
Approximating the multiple stochastic integral leads to
\begin{equation}\label{eq:Hermite-as-tetrahedral-polynomial}
  \begin{split}
  h_d(g) &= d! \lim_{N \to \infty} N^{-d/2} \sum_{1 \le j_1 < \cdots < j_d \le N} g_{j_1, N} \cdots g_{j_d, N} \\[1ex]
         &= \lim_{N \to \infty} N^{-d/2} \sum_{j \in [N]\uu{d}} g_{j_1, N} \cdots g_{j_d, N},
\end{split}
\end{equation}
where the limit is in $L^2(\Omega)$ (see \cite[Theorem 7.3. and formula (7.9)]{JansonGHS}) and actually the convergence holds in any $L^p$ (see~\cite[Theorem 3.50]{JansonGHS}).

Now, consider $n$ independent copies $(W_t\ub{i})_{t \in [0,1]}$ of the Brownian motion ($i=1, \ldots, n$) together with the corresponding Gaussian random variables: $g\ub{i} = W_1\ub{i}$ and, for $N \ge 1$,
\[
  g_{j, N}\ub{i} = \sqrt{N} (W_{\frac{j}{N}}\ub{i} - W_{\frac{j-1}{N}}\ub{i}), \quad j = 1, \ldots, N.
\]
Let also
\[
  G\ub{n, N} = (g_{1,N}\ub{1}, \ldots, g_{N,N}\ub{1}, \ g_{1,N}\ub{2}, \ldots, g_{N,N}\ub{2}, \ \ldots,\  g_{1,N}\ub{n}, \ldots, g_{N,N}\ub{n}) = (g_{j,N}\ub{i})_{(i,j)\in [n]\times [N]}
\]
be a Gaussian vector with $n \times N$ coordinates. We identify here the set $[nN]$ with $[n]\times [N]$ via the bijection $(i,j) \leftrightarrow (i-1)N+j$. We will also identify the sets $([n]\times [N])^d$ and $[n]^d\times [N]^d$ in a natural way. For $d \ge 0$ and $\rd \in \Delta_d^n$, let
\[
  I_{\rd} = \big\{ i \in [n]^d \colon \forall_{l \in [n]} \ \# i^{-1}(\{l\}) = d_l \big\},
\]
and define a $d$-indexed matrix $B_{\rd}\ub{N}$ of $n^d$ blocks each of size $N^d$ as follows: for $i \in [n]^d$ and $j \in [N]^d$,
\[
  \big(B_{\rd}\ub{N}\big)_{(i, j)} = \begin{cases}
    \frac{d_1! \cdots d_n!}{d!} N^{-d/2} & \text{if $i \in I_{\rd}$ and $(i, j) := \big((i_1, j_1), \ldots, (i_d, j_d)\big) \in ([n] \times [N])\uu{d},$}
    \\[1ex]
    0 & \text{otherwise.}
  \end{cases}
\]

\begin{proof}[Proof of Proposition \ref{prop:splitting-derivatives}]
Assume that $f$ is of the form~\eqref{eq:f-as-Hermite},
By \cite[Lemma 4.3]{AW}, for any $p > 0$,

\[
  \big\langle B_{\rd}\ub{N}, (G\ub{n,N})^{\otimes d} \big\rangle \stackrel{N \to \infty}{\longrightarrow} h_{d_1}(g\ub{1}) \cdots h_{d_n}(g\ub{n}) \quad \text{in $L^p(\Omega)$},
\]
which together with the triangle inequality implies that
\[
  \lim_{N \to \infty} \Big\|\sum_{d=1}^D \Big\langle \sum_{\rd \in \Delta_d^n} a_\rd B_\rd\ub{N}, \big(G\ub{n,N}\big)^{\otimes d} \Big\rangle \Big\|_p = \big\|f(G) - \E f(G)\big\|_p
\]
for any $p > 0$, where $G = (g\ub{1}, \ldots, g\ub{n} )$ and we interpret multiplication of an element  of $F$ and a real valued $d$ indexed matrix in a natural way.
Thus, by Theorem \ref{thm_Kwapien} and the triangle inequality we obtain
\begin{multline*}
C_D^{-1}  {\lim}_{N\to \infty} \sum_{d=1}^D \Big\| \Big\langle \sum_{\rd \in \Delta_d^n} a_\rd B_\rd\ub{N}, \big(G\ub{n,N}\big)^{\otimes d} \Big\rangle \Big\|_p\\
  \le \|f(G)-\E f(G)\|_p \\
  \le  {\lim}_{N\to \infty}  \sum_{d=1}^D \Big\| \Big\langle \sum_{\rd \in \Delta_d^n} a_\rd B_\rd\ub{N}, \big(G\ub{n,N}\big)^{\otimes d} \Big\rangle \Big\|_p
\end{multline*}
(recall that the matrices $B_\rd\ub{N}$ have zeros on generalized diagonals and so do their linear combinations).

Denote by $G\ub{n,N,1}, \ldots, G\ub{n,N,d}$ independent copies of $G\ub{n,N}$.


By decoupling inequalities of Theorem \ref{thm:decoupling} we have
\begin{align}\label{eq:deoupling-the-blown-up-thing}
\Big\| \Big\langle \sum_{\rd \in \Delta_d^n} a_\rd B_\rd\ub{N}, \big(G\ub{n,N}\big)^{\otimes d} \Big\rangle \Big\|_p \sim^d \Big\| \Big\langle \sum_{\rd \in \Delta_d^n} a_\rd B_\rd\ub{N}, G\ub{n,N,1}\otimes \cdots \otimes  G\ub{n,N,d} \Big\rangle \Big\|_p.
\end{align}

To finish the proof it is therefore enough to show that for any $d\le D$,
\begin{multline}\label{eq:description-via-diff}
  \lim_{N \to \infty} \Big\| \Big\langle \sum_{\rd \in \Delta_d^n} a_\rd B_\rd\ub{N}, G\ub{n,N,1}\otimes \cdots \otimes  G\ub{n,N,d} \Big\rangle \Big\|_p
  =   \frac{1}{d!} \| \langle A_d,G_1\otimes\cdots\otimes G_d\rangle\|_p,
\end{multline}
where $G_1,\ldots,G_D$ are independent copies of $G$.

Fix $d \ge 1$. For any $\rd \in \Delta_d^n$ define a symmetric $d$-indexed matrix $(b_\rd)_{i \in [n]^d}$ as
\[
  (b_\rd)_i = \begin{cases}
    \frac{d_1! \cdots d_n!}{d!} & \text{if $i \in I_\rd,$} \\
    0 & \text{otherwise.}
  \end{cases}
\]
and a symmetric $d$-indexed matrix $(\tilde{B}_\rd\ub{N})_{(i, j) \in ([n] \times [N])^d}$ as
\[
  (\tilde{B}_\rd\ub{N})_{(i, j)} = N^{-d/2} (b_\rd)_i \quad \text{for all $i \in [n]^d$ and $j \in [N]^d.$}
\]
Using the convolution properties of Gaussian distributions one easily obtains
\begin{equation}\label{eq:blown-matrices}
  \Big\|\Big\langle  \sum_{\rd \in \Delta_d^n} a_\rd \tilde{B}_\rd\ub{N},  G\ub{n,N,1}\otimes \cdots \otimes  G\ub{n,N,d} \rangle \Big\|_p = \Big\| \Big\langle \sum_{\rd \in \Delta_d^n} a_\rd  (b_\rd)_{i \in [n]^d}, G_1\otimes \cdots \otimes G_d\Big\rangle \Big\|_p
\end{equation}

On the other hand, for any $\rd \in \Delta_d^n$, the matrices $\tilde{B}_\rd\ub{N}$ and $B_\rd\ub{N}$ differ at no more than $|I_\rd | \cdot |([N]^d \setminus [N]\uu{d})|$ entries. Thus
\begin{align*}
  & \Big\| a_\rd \Big\langle \tilde{B}_\rd\ub{N} - B_\rd\ub{N},G\ub{n,N,1}\otimes \cdots \otimes  G\ub{n,N,d}\Big \rangle \Big\|_p \\
  & \le C(d)p^{\frac{d}{2}} \|a_\rd\| \cdot \Big\|\Big\langle \tilde{B}_\rd\ub{N} - B_\rd\ub{N},G\ub{n,N,1}\otimes \cdots \otimes  G\ub{n,N,d}\Big \rangle \Big\|_2 \\
  &\le  p^{\frac{d}{2}} \|a_\rd\|\cdot \sqrt{ |I_\rd| \Big(\frac{d_1! \cdots d_n!}{d!}\Big)^2 N^{-d} \Big(N^d -\frac{N!}{(N-d)!}\Big) }\longrightarrow 0
\end{align*}
as $N \to \infty$, where in the first inequality we used Theorem \ref{gaushiper}.

Together with the triangle inequality and \eqref{eq:blown-matrices} this gives
\begin{equation}\label{eq:B-and-b}
\lim_{N \to \infty} \Big\| \Big\langle \sum_{\rd \in \Delta_d^n} a_\rd B_\rd\ub{N}, G\ub{n,N,1}\otimes \cdots \otimes  G\ub{n,N,d} \Big\rangle \Big\|_p
=  \Big\| \Big\langle \sum_{\rd \in \Delta_d^n} a_\rd  (b_\rd)_{i \in [n]^d}, G_1\otimes \cdots \otimes G_d\Big\rangle \Big\|_p.
\end{equation}

Finally, we have
\begin{equation}\label{eq:Df-and-b}
  \E \D^d f(G) = d! \sum_{\rd \in \Delta_d^n} a_\rd  (b_\rd)_{i \in [n]^d}.
\end{equation}
Indeed, using the identity on Hermite polynomials, $\frac{d}{dx} h_k(x) = k h_{k-1}(x)$ ($k \ge 1$), we obtain $\E \frac{d^l}{dx^l} h_k(g) = k! \textbf{1}_{k=l}$ for $k,l\ge 0$,
and thus, for any $d, l \le D$ and $\rd \in \Delta_l^n$,
\[
  \big(\E \D^d h_{d_1}(g\ub{1}) \cdots h_{d_n}(g\ub{n})\big)_i = d! (b_\rd)_i \textbf{1}_{d=l} \quad \text{for each $i \in [n]^d$}.
\]
Now \eqref{eq:Df-and-b} follows by linearity. Combining it with~\eqref{eq:B-and-b} yields~\eqref{eq:description-via-diff} and ends the proof.
\end{proof}

\subsection{ Proof of a bound for exponential chaoses}
\begin{proof}[Proof of Proposition \ref{wyk}]

Lemma \ref{lem95} implies
\begin{align}
\lv \sum_\ii a_\ii \prod_{k=1}^d E_{i_k}^{k} \rv_p \sim^d \lv \sum_{i_1,\ldots,i_{2d}} \hat{a}_{i_1,\ldots,i_{2d}} \prod_{k=1}^{2d} g_{i_k}^{k}  \rv_p, \label{loc1}
\end{align}
where
$$
\hat{a}_{i_1,\ldots,i_{2d}}:=a_{\id}\1_{\{i_1=i_{d+1},\ldots,i_d=i_{2d}\}}.
$$
Let
$\hat{A}=(\hat{a}_{i_1,\ldots,i_{2d}})_{i_1,\ldots,i_{2d}}.$
Theorem \ref{Lq} and \eqref{loc1} yield
\begin{align}
\label{loc2}
\frac{1}{C(d)}q^{1/2- d} \sum_{J\subset [2d]} \sum_{\pp\in \pp([J])} p^{\frac{|\pp|}{2}}\altnor{\hat{A}}^{L_q}
&\leq \lv \sum_\ii a_\ii \prod_{k=1}^d E_{i_k}^{k} \rv_p \\
&\leq C(d) q^{2d-\frac{1}{2}} \sum_{J\subset [2d]} \sum_{\pp\in \pp([J])} p^{\frac{|\pp|}{2}}\altnor{\hat{A}}^{L_q}. \nonumber
\end{align}

We will now express $\sum_{J\subset [2d]} \sum_{\pp\in \pp([J])} p^{\frac{|\pp|}{2}}\altnor{\hat{A}}^{L_q}$ in terms of the matrix $A$. To this end we need to introduce new notation. Consider a finite sequence $\mathcal{M} = (J,I_1,\ldots,I_k)$ of subsets of $[d]$, such that $J\cup I_1\cup\ldots\cup I_k= [d]$, $I_1,\ldots,I_k \neq \emptyset$ and each number $m\in [d]$ belongs to at most two of the sets $J,I_1,\ldots,I_k$. Denote the family of all such sequences by $\mm([d])$. For $\mathcal{M}=(J,I_1,\ldots,I_k)$ set $|\mm| = k+1$ and
$$
\robnn{A}^{L_q}_{\mm}:= \sup \left\{\lv  \sqrt{\sum_{i_{J}} \left(\sum_{i_{[d] \setminus J}} a_\ii \prod_{r=1}^k x^r_{i_{I_r}} \right)^2} \rv_{L_q}  \ \Big{|} \ \forall_{r\leq k} \sum_{i_{I_r}} \left(  x^r_{i_{I_r}} \right)^2 \leq 1\right\},
$$
where we do not exclude that $J=\emptyset$. By a straightforward verification
\begin{align}
\sum_{J\subset [2d]} \sum_{\pp\in \pp([J])} p^{\frac{|\pp|}{2}}\altnor{\hat{A}}^{L_q}\sim^d\sum_{\mm \in \mm([d])} p^{\frac{|\mm|-1}{2}}\robnn{A}^{L_q}_{\mm}. \label{loc22}
\end{align}
To finish the proof it is enough to show that
\begin{equation}
\sum_{\mm \in \mm([d])} p^{\frac{|\mm|-1}{2}}\robnn{A}^{L_q}_{\mm}
\sim^d \sum_{\mm \in \ccc} p^{\frac{|\mm|-1}{2}}\robnn{A}^{L_q}_{\mm},\label{suma}
\end{equation}
where
\begin{align*}
\ccc=\left\{\mm=(J,I_1,\ldots,I_k) \in \mm([d]) \ \Big{|} \ J\cap \left(\bigcup_{l=1}^k I_l\right)=\emptyset,\ \right.
\\
\left. \forall_{l,m \leq k}\   I_m \cap I_l \neq \emptyset \Rightarrow (|I_l|=|I_m|=1,\ I_l=I_m)  \vphantom{\left(\bigcup_{l=1}^k I_l\right)} \right\}.
\end{align*}

Indeed assume that \eqref{suma} holds and choose $\mm=(J,I_1,\ldots,I_k) \in \ccc$.
Consider $I=\{i \ | \ \exists_{l<m\leq k} \ \{i\}=I_l=I_m \}$. Then  $J\cap I=\emptyset$ and we have
\begin{align}
\label{al}
&\left(\robnn{A}^{L_q}_{\mm }\right)^q=\sup \left\{ \int_X \left(\sum_{i_J}  \left(\sum_{i_{J^c}} a_\id  \prod_{l \in I} y^l_{i_l} x^l_{i_l}   \prod_{\substack{l\leq k \\ I_l\cap I = \emptyset}} x^l_{i_{I_l}}  \right)^2  \right)^{q/2}  d\mu(x) \ \Big{|} \right.
\\
&\phantom{aaaaaaaaaaaaaa}\left.   \forall_{1\leq l\leq k} \sum_{i_{I_l}} (x^l_{i_{I_l}})^2\leq 1,\
\forall_{l \in I}\ \sum_{i_l} (y^l_{i_l})^2\leq 1    \right\}\nonumber\\
&=\max_{i_I} \sup \left\{ \int_X \left(\sum_{i_J}  \left(\sum_{i_{J^c\setminus I}} a_\id    \prod_{\substack{l\leq k \\ I_l\cap I = \emptyset}} x^l_{i_{I_l}}  \right)^2  \right)^{q/2}  d\mu(x) \ \Big{|} \  \forall_{\substack{1\leq l\leq k \\ I_l \cap I=\emptyset }}\ \sum_{i_{I_l}} (x^l_{i_{I_l}})^2\leq 1    \right\}\nonumber
\\
&=\max_{i_I}\left(\robn{(a_{i_1,\ldots,i_d})_{i_{I^c}}}^{L_q}_{\{I_l \ : \ I_l\cap I = \emptyset\}}\right)^q=: \max_{i_I}\left(\robn{(a_{i_1,\ldots,i_d})_{i_{I^c}}}^{L_q}_{\pp}\right)^q,\nonumber
\end{align}

where in the second equality we used the fact that $(y^l_{i_l} x^l_{i_l} )_{i_l}\in B^n_1$ together with convexity and homogeneity of the norm
\begin{equation}
\|(f_{i_J})_{i_J}\|_{L_q(\ell_2)} = \left(\int_X \left(\sum_{i_J} f_{i_J}^2\right)^{q/2}\right)^{1/q}. \nonumber
\end{equation}

By combining the above with \eqref{loc2}-\eqref{suma} we conclude the assertion of the proposition.

The proof is completed by showing that
$$\sum_{\mm \in \mm([d])} p^{\frac{|\mm|-1}{2}}\robnn{A}^{L_q}_{\mm}\leq C(d) \sum_{\mm \in \ccc} p^{\frac{|\mm|-1}{2}}\robnn{A}^{L_q}_{\mm} $$
 (the second inequality in \eqref{suma} is trivial), which will be done in two steps. Let us fix $\mm=(J,I_1,\ldots,I_k) \in \mm([d])$.

\begin{enumerate}
\item Assume first that $J\cap (\bigcup_{i=1}^k I_i ) \neq \emptyset$.
Without loss of generality  we can assume that $1 \in J\cap I_1$. Denote $\hat{I}_1=I_1\setminus \{1\}$ and for any matrix
$(x_{i_I}^1)_{i_I}$ such that $\sum_{i_{I_1}} (x_{i_{I_1}}^1)^2\le 1$, set
$(b_{i_1})_{i_1}:=(\sqrt{\sum_{i_{I_1\setminus \{1\}}} (x^1_{i_{I_1}})^2})_{i_1}$. Clearly,
$$
\left( \left(b_{i_1} \right)^2 \right)_{i_1} \in B^n_1\quad \textrm{ and }\quad \sum_{i_{I_1 \setminus \{1\}}} \left(\frac{x^1_{i_{I_1}}}{b_{i_1}} \right)^2 \leq 1.
$$
 Observe that  for any  $f_1,\ldots,f_n \in L^q(X,d\mu)$ the function
$$
[0,+\infty)^n \ni v \rightarrow \int_X \left( \sum_i f^2_i(x) v_i\right)^{q/2} d\mu(x)
$$
is convex (recall that $q\geq 2$). Therefore, we have
\begin{align*}
&\left(\robnn{A}^{L_q}_{\mm}\right)^{q}=  \sup \left\{ \int_X \left(\sum_{i_J} \left(b_{i_1}\right)^2 \left(\sum_{i_{J^c}} a_\id \frac{x^1_{i_{I_1}}}{b_{i_1}}    \prod_{l=2}^k x^l_{i_{I_l}}  \right)^2  \right)^{q/2}  d\mu(x) \ \Big{|} \right.
\\
&\phantom{aaaaaaaaaaaaaaaa}\left. \  \forall_{1\leq l\leq k} \sum_{i_{I_l}} (x^l_{i_{I_l}})^2\leq 1    \right\} \\
&\le\max_{i_1}\sup \left\{  \int_X \left(\sum_{i_{J\setminus\{1\}}} \left(\sum_{i_{J^c}} a_\id \frac{x^1_{i_{I_1}}}{b_{i_1}}    \prod_{l=2}^k x^l_{i_{I_l}}  \right)^2  \right)^{q/2}  d\mu(x) \ \Big{|} \right.
\\
&\phantom{aaaaaaaaaaaaaaaa}\left. \  \forall_{1\leq l\leq k} \sum_{i_{I_l}} (x^l_{i_{I_l}})^2\leq 1    \right\} \\
&\leq \max_{i_1} \sup \left\{ \int_X \left(\sum_{i_{J\setminus \{1\}}} \left(\sum_{i_{J^c}} a_\id y_{i_{\hat{I}_1} }   \prod_{l=2}^k x^l_{i_{I_l}}  \right)^2  \right)^{q/2}  d\mu(x) \ \Big{|} \right.
\\
&\phantom{aaaaaaaaaaaaaaaa}\left. \sum_{i_{\hat{I}_1}} (y_{i_{\hat{I}_1}})^2\leq 1,\ \forall_{1\leq l\leq k} \sum_{i_{I_l}} (x^l_{i_{I_l}})^2\leq 1    \right\}.
\end{align*}
If $\hat{I}_1\neq \emptyset$ let $\mathcal{M}' = (J\setminus\{1\},\{1\},\{1\},\hat{I}_1,I_2,\ldots,I_k)$, otherwise
set $\mathcal{M}' = (J\setminus\{1\},\{1\},\{1\},I_2,\ldots,I_k)$.

By the same argument as was used for the second equality in \eqref{al} we obtain that the right-hand side above equals $\robnn{A}^{L_q}_{\mm '}$, which gives

\begin{displaymath}
\robnn{A}^{L_q}_{\mm}\le \robnn{A}^{L_q}_{\mm '}.
\end{displaymath}

Observe that
$$
p^{(|\mm|-1)/2}\robnn{A}^{L_q}_{\mm} \leq p^{(|\mm|-1)/2}\robnn{A}^{L_q}_{\mm '}
\leq p^{(|\mm '|-1)/2}\robnn{A}^{L_q}_{\mm '}.
$$
By iterating this argument we obtain that $p^{(|\mm|-1)/2}\robnn{A}^{L_q}_{\mm}\le p^{(|\mm''|-1)/2}\robnn{A}^{L_q}_{\mm''}$ for some $\mm'' = (J'',I_1'',\ldots,I_m'')$ such that $J''\cap (\bigcup_{l=1}^m I_l'')=\emptyset$.

\item Assume that for some $l,m\leq k$ $I_l \cap I_m \neq \emptyset$ and $|I_l| \geq 2$ or $|I_m|\geq 2$. \\
Without loss of the generality $1 \in I_1 \cap I_2$ and $|I_1|\geq 2$. Clearly,
\begin{align*}
&\left(\robnn{A}^{L_q}_{\mm}\right)^{q}=  \sup \left\{ \int_X \left(\sum_{i_J} \left(\sum_{i_{J^c}} a_\id b_{i_1}c_{i_1}  \frac{x^1_{i_{I_1}}}{b_{i_1}} \frac{x^2_{i_{I_2}}}{c_{i_1}}  \prod_{l=3}^k x^l_{i_{I_l}}  \right)^2  \right)^{q/2}  d\mu(x) \ \Big{|} \right.
\\
&\phantom{aaaaaaaaaaaaaaaaaaaaaaa} \left.  \forall_{1\leq l\leq k} \sum_{i_{I_l}} (x^l_{i_{I_l}})^2\leq 1    \right\},
\end{align*}
where  $(b_{i_1})_{i_1}:=(\sqrt{\sum_{i_{I_1\setminus \{1\}}} (x^1_{i_{I_1}})^2})_{i_1},\ (c_{i_1})_{i_1}:=(\sqrt{\sum_{i_{I_2\setminus \{1\}}} (x^2_{i_{I_2}})^2})_{i_1} \in B^n_2$.
Since  $(b_{i_1}c_{i_1})_{i_1}\in B_1^n$,
$$
\forall_{i_1}\sum_{i_{I_1\setminus \{1\}}}  \left(\frac{x^1_{i_{I_1}}}{b_{i_1}} \right)^2 \leq 1,\
\sum_{i_{I_2\setminus \{1\}}}  \left(\frac{x^2_{i_{I_2}}}{ c_{i_1}} \right)^2 \leq 1,
$$
and for any $(f_{ij})_{ij}$ in $L^q(X,d\mu)$, the function
$$
\R^n \ni v \rightarrow \int_X \left( \sum_i \left(\sum_j v_j f_{ij}(x) \right)^2 \right)^{q/2} d\mu(x)
$$
is convex, we obtain similarly as in step 1,
$$
p^{(|\mm|-1)/2}\robnn{A}^{L_q}_{\mm}\leq p^{(|\mm|-1)/2}\ \robnn{A}^{L_q}_{\mm '}
\leq p^{(|\mm '|-1)/2}\ \robnn{A}^{L_q}_{\mm '}
$$
where $\mm' :=(J,\{1\},\{1\},I_1\setminus\{1\},I_2 \setminus \{1\},I_3,\ldots,I_k)$ if $I_2\setminus\{1\}\neq \emptyset$ and $\mm' :=(J,\{1\},\{1\},I_1\setminus\{1\},I_3,\ldots,I_k)$ otherwise.
An iteration of this argument shows that indeed one can assume that $\mathcal{M}$ satisfies the implication $I_m \cap I_l \neq \emptyset \Rightarrow (|I_l|=|I_m|=1,\ I_l=I_m)$.
\end{enumerate}

Combining Steps $1$ and $2$ we obtain that for any $\mm \in \mm([d])$ there exists $\mm ' \in \ccc$ such that   $p^{(|\mm|-1)/2}\robnn{A}^{L_q}_{\mm} \leq p^{(|\mm '|-1)/2}\ \robnn{A}^{L_q}_{\mm '}$ which yields \eqref{suma}.
\end{proof}

\appendix

\section{}
In this section we gather technical facts that are used in the proof.

\begin{twr}[Hypercontractivity of Gaussian chaoses]\label{gaushiper}
Let
$$S=a+\sum_{i_1} a_{i_1} g_{i_1}+\sum_{i_1,i_2} a_{i_1,i_2} g_{i_1} g_{i_2}+\ldots+\sum_{i_1,\ldots,i_d} a_{i_1,\ldots,i_d} g_{i_1} \cdots g_{i_d},$$
be a non-homogeneous Gaussian chaos of order $d$ with values in a Banach space $(F,\lv \cdot \rv)$. Then for any $1\leq p<q<\infty$, we have
$$\left( \E \lv S \rv^q \right)^{1/q} \leq C(d) \left(\frac{q}{p} \right)^{d/2} \left(\E \lv S \rv^p \right)^{1/p}.$$
\end{twr}
\begin{proof}
It is an immediate consequence of \cite[Theorem 3.2.10]{7} and H{\"o}lder's inequality.
\end{proof}

\begin{twr}[Sudakov minoration {\cite{Su}}]\label{norsud}
For any set $T\subset \R^n$ and $\eps>0$ we have
$$
\eps \sqrt{\ln N(T,d_2,\eps )} \leq C \E \sup_{t \in T} \sum_i t_i g_i,
$$
where $d_2$ is the Euclidean distance.
\end{twr}

\begin{twr}[Dual Sudakov minoration {\cite[formula (3.15)]{15}}]
\label{thm:dualSud}
Let $\alpha$ be a norm on $\R^n$ and $\rho_\alpha(x,y)=\alpha(x-y)$ for $x,y\in \R^n$.
Then
\[
\eps\sqrt{\log N(B_2^n,\rho_\alpha,\eps)}\leq C\E\alpha(G_n)\quad \mbox{for }\eps>0.
\]
\end{twr}

\begin{lem}\cite[Lemma 3]{Latgaus}\label{rozbij}
Let $(G_t)_{t \in T}$ be a centered Gaussian process and $T = \bigcup_{l=1}^m T_l$, $m \geq 1$. Then
$$
\E \sup_{t \in T} G_t
\leq \max_{l \leq m} \E \sup_{t \in T_l} G_t + C\sqrt{\ln(m)} \sup_{t,t' \in T} \sqrt{\E (G_t-G_{t'})^2}.
$$
\end{lem}

\begin{lem}\cite[Lemma 4]{Latgaus}\label{kon}
Let $G$ be a Gaussian variable in a normed space $(F, \lv \ \cdot \  \rv)$. Then for any $p\geq 2$,
$$
\frac{1}{C} \left( \lv G \rv_1+\sqrt{p}\sup_{\substack{\varphi \in F^*\\ \lv \varphi\rv_* \leq 1}} \E | \varphi(G) | \right)
\leq \lv G  \rv_p
\leq \lv G \rv_1+C\sqrt{p} \sup_{\substack{\varphi \in F^*\\ \lv \varphi\rv_* \leq 1}} \E | \varphi(G) |,
$$
where $(F^*,\lv \cdot \rv_*)$ is the dual space to $(F,\lv \cdot \rv)$.
\end{lem}

\begin{twr}\cite[Theorem 1]{Latgaus}\label{momrealgaus}
For any real-valued matrix $(a_{i_1,\ldots,i_d})_{i_1,\ldots,i_d}$ and $p \geq 2$, we have
$$\lv \sum_{\ii} a_{\i} \prod_{j=1}^d g^j_{i_j} \rv_p \sim ^d \sum_{\stackrel{\mathcal{P} \in \mathcal{P}([d])}{\mathcal{P}=(I_1,\ldots,I_k)}} p^{|P|/2} \sup \left\{ \sum_{\ii} a_{\ii} \prod_{j=1}^k x^j_{i_{I_j}} \ \mid \ \lv (x^k_{i_{I_k}})_{i_{I_k}} \rv_2 \leq 1 \right\}.$$
\end{twr}

\begin{cor}\label{lgaus}
Assume that for any $ i_1,\ldots,i_d,\ a_{i_1,\ldots,i_d} \in \R$. Then for all $p\geq 1$
$$
\frac{1}{C(d)} \sqrt{p}\sqrt{\sum_{i_1,\ldots,i_d} a^2_{i_1,\ldots,i_d}}
\leq \lv \sum_{i_1,\ldots,i_d} a_{i_1,\ldots,i_d} \pg \rv_p
\leq C(d) p^{d/2} \sqrt{\sum_{i_1,\ldots,i_d} a^2_{i_1,\ldots,i_d}}.
$$
\end{cor}
\begin{proof}
It is an easy consequence of Theorems \ref{gaushiper} and \ref{momrealgaus}.
\end{proof}

\begin{lem}\cite[Lemma 9.5]{wyk} \label{lem95}
Let $Y^{(1)}_i$ be independent standard symmetric exponential variables (variables with density $2^{-1}\exp(-|t|)$) and $Y^{(2)}_i=g^2_i$, $Y^{(3)}_i= g_ig'_i$, where
$g_i,g'_i$ are i.i.d. $\mathcal{N}(0,1)$ variables and $\eps_i$ – i.i.d. Rademacher variables independent of $(Y^{(1)}),\ (Y^{(2)}),\ (Y^{(3)})$. Then for any Banach space $(F,\lv \cdot \rv)$ and any vectors $v_1,\ldots,v_n \in F$ the quantities
$$
\E \lv \sum_i v_i \eps_i Y^{(j)}_i \rv, \quad j=1,2,3,
$$
are comparable up to universal multiplicative factors.
\end{lem}

We remark that the above lemma is formulated in \cite{wyk} in the real valued case, however the proof presented there (based on the contraction principle) works in arbitrary Banach spaces.

We will also need decoupling inequalities for tetrahedral homogeneous polynomials. Such inequalities were introduced for the first time in \cite{McConnell-Taqqu} for real valued multi-linear forms and since then have been strengthened and generalized by many authors (see the monograph \cite{7}). The following theorem is a special case of results from \cite{kwa} (treating also general tetrahedral polynomials) and \cite{delaPena-decoupling, pen} (treating general $U$-statistics).

\begin{twr}\label{thm:decoupling}  Let $X = (X_1,\ldots,X_n)$ be a sequence of independent random variables and let $X^l = (X^l_1,\ldots,X^l_n)$, $l=1,\ldots,d$,  be i.i.d. copies of $X$. Consider a $d$-indexed symmetric matrix $(a_{i_1,\ldots,i_d})_{i_1,\ldots,i_d=1}^n$, with coefficients from a Banach space $F$. Assume that $a_{i_1,\ldots,i_d} = 0$ whenever there exist $1 \le k< m \le d$ such that $i_k = i_m$. Then for any $p \ge 1$,
\begin{displaymath}
  \Big\|\sum_{i_1,\ldots,i_d = 1}^n a_{i_1,\ldots,i_d}X_{i_1}\cdots X_{i_d}\Big\|_p \sim_{d} \Big\|\sum_{i_1,\ldots,i_d = 1}^n a_{i_1,\ldots,i_d} X^{1}_{i_1}\cdots X^{d}_{i_d}\Big\|_p.
\end{displaymath}
\end{twr}

\end{document}